\documentclass[11pt]{smfart}
\usepackage{color}
\usepackage{amssymb,verbatim}
\usepackage{hyperref}
\usepackage{amsthm,array,amssymb,amscd,amsfonts,amssymb,latexsym, url}
\usepackage{amsmath}
\usepackage[all]{xy}
\usepackage[french]{babel}

\usepackage{cyrillic}

\newcommand{\cyrbf}{\fontencoding{OT2}\selectfont\textcyrbf}

\newcounter{spec}
{\end{list}}

\renewcommand{\P}{{\mathbf P}}

\newcommand{\N}{{\mathbb N}}
\newcommand{\Z}{{\mathbb Z}}
\newcommand{\Q}{{\mathbb Q}}
\newcommand{\C}{{\mathbb C}}
\newcommand{\R}{{\mathbb R}}

\renewcommand{\L}{{\mathbb L}}

\newcommand{\Br}{{\operatorname{Br  }}}

\newcommand{\Spec}{{\operatorname{Spec \ }}}

\newcommand{\Norm}{{\operatorname{Norm}}}

\renewcommand{\lim}{\varprojlim}

\numberwithin{equation}{section}

\newfont{\gothic}{eufb10}

\renewcommand{\qed}{{\hfill$\square$}}

\newtheorem{theo}{Th\'{e}or\`{e}me}[section]
\newtheorem{prop}[theo]{Proposition}

\newtheorem{lem}[theo]{Lemme}
\newtheorem{cor}[theo]{Corollaire}
\theoremstyle{definition}
\newtheorem{defi}[theo]{D\'efinition}
\theoremstyle{remark}
\newtheorem{rema}[theo]{Remarque}

\newtheorem{ex}[theo]{Exemple}

\theoremstyle{definition}
\newtheorem{thmd}[theo]{Th\'eor\`eme}

\newcommand{\bthe}{\begin{theo}}
\newcommand{\ble}{\begin{lem}}
\newcommand{\bpr}{\begin{prop}}
\newcommand{\bco}{\begin{cor}}
\newcommand{\bde}{\begin{defi}}
\newcommand{\ethe}{\end{theo}}
\newcommand{\ele}{\end{lem}}
\newcommand{\epr}{\end{prop}}
\newcommand{\eco}{\end{cor}}
\newcommand{\ede}{\end{defi}}

\newcommand{\Pic}{\operatorname{Pic}}
\newcommand{\Div}{\operatorname{Div}}

\newcommand{\F}{{\mathbb F}}

\newcommand{\G}{{\mathbb G}}

\DeclareFontFamily{U}{wncy}{}
\DeclareFontShape{U}{wncy}{m}{n}{%
<5>wncyr5%
<6>wncyr6%
<7>wncyr7%
<8>wncyr8%
<9>wncyr9%
<10>wncyr10%
<11>wncyr10%
<12>wncyr6%
<14>wncyr7%
<17>wncyr8%
<20>wncyr10%
<25>wncyr10}{}
\DeclareMathAlphabet{\cyr}{U}{wncy}{m}{n}

\begin{document}

  \title[Hypersurfaces quartiques de dimension 3]
 {Hypersurfaces quartiques   de dimension 3 : Non rationalit\'e stable}

\author{J.-L. Colliot-Th\'el\`ene}
\address{C.N.R.S., Universit\'e Paris Sud\\Math\'ematiques, B\^atiment 425\\91405 Orsay Cedex\\France}
\email{jlct@math.u-psud.fr}
\author{A. Pirutka}
\address{C.N.R.S., \'Ecole Polytechnique, CMLS,  91128 Palaiseau\\France}
\email{alena.pirutka@polytechnique.edu}

\date{\`A para\^{\i}tre dans  Ann. Sc. \'Ec. Norm. Sup.;
soumis le 27 mars 2014;  rapports re\c cus  par les auteurs le 12 novembre 2014;
version r\'evis\'ee soumise le 13 janvier  2015;  article accept\'e le 17 mars 2015.}
\maketitle

  \begin{abstract}
Inspir\'es par un argument de C. Voisin, nous montrons l'existence d'hypersurfaces quartiques lisses de dimension 3 sur les complexes qui ne sont pas stablement rationnelles, plus pr\'ecis\'ement dont le groupe de Chow de degr\'e z\'ero n'est pas universellement \'egal \`a  $\Z$.
La m\'ethode de sp\'ecialisation adopt\'ee ici permet de construire des exemples d\'efinis sur un corps de nombres.
  \end{abstract}

  \begin{altabstract}  There are (many) smooth quartic threefolds over the complex field which are not stably rational. More precisely, their degree zero Chow group is not universally equal to $\Z$. The proof uses a variation of a method due to C. Voisin. The specialisation argument we use yields
examples defined over a number field.
  \end{altabstract}

\section*{Introduction}

Soit $X \subset \P^4_{\C}$ une hypersurface quartique lisse.
Dans \cite{iskovkikhmanin}, Iskovskikh et Manin  montrent que tout automorphisme birationnel de $X$
est un automorphisme, ce qui implique que le groupe des automorphismes birationnels est fini
et que la vari\'et\'e de Fano $X$ n'est pas rationnelle. Des choix convenables de $X$ donnent alors
des contre-exemples
au th\'eor\`eme de L\"{u}roth pour les solides.

 Cette m\'ethode, dite de rigidit\'e 
 birationnelle, 
 a depuis \'et\'e fort d\'evelopp\'ee.
 Elle ne permet pas de r\'epondre \`a la question de la rationalit\'e
stable de ces vari\'et\'es, que l'on trouve pos\'ee explicitement
dans  \cite{huh}.

Artin et Mumford \cite{artinmumford} construisirent d'autres exemples de 
solides $X/\C $ projectifs et lisses   qui sont unirationnels mais non rationnels.
L'invariant qu'ils utilis\`erent est le sous-groupe de torsion $H^3(X,\Z)_{tors}$  
du troisi\`eme groupe de cohomologie de Betti,
isomorphe pour un solide projectif et lisse  au groupe  $H^4(X,\Z)_{tors}$.
Pour toute  vari\'et\'e $X/\C$  projective et lisse
 rationnellement connexe,
 le  groupe $H^3(X,\Z)_{tors}$ est isomorphe \`a un autre
invariant birationnel, le groupe de Brauer $\Br (X)$. Ce groupe est nul pour toute vari\'et\'e $X/\C$ 
stablement rationnelle, et m\^eme pour toute vari\'et\'e r\'etracte rationnelle.
Leurs exemples ne sont donc pas stablement rationnels.
La m\'ethode ne peut s'appliquer directement aux vari\'et\'es  intersections compl\`etes lisses de dimension
au moins 3 dans un espace projectif $\P^n_{\C}$, car
le groupe de Brauer de telles vari\'et\'es  est nul.

Dans un r\'ecent article \cite{voisin}, C. Voisin a montr\'e qu'un solide lisse rev\^etement double
de $\P^3_{\C}$  ramifi\'e le long
d'une  surface quartique lisse tr\`es  g\'en\'erale n'est pas stablement rationnel. Elle utilise
une famille propre $f : X \to B$ de vari\'et\'es, de base une courbe $B$ lisse, d'espace total une vari\'et\'e lisse $X$,
dont une fibre sp\'eciale $Y$ est   un solide d'Artin-Mumford, de d\'esingularisation $Z\to Y$. 
 Utilisant le fait que le solide $Y$ n'a que des singularit\'es quadratiques ordinaires, par un argument
 de sp\'ecialisation, elle montre
 que si une fibre tr\`es g\'en\'erale de $f$ admettait une d\'ecomposition de Chow de la diagonale,
 alors il en serait de m\^eme pour la vari\'et\'e lisse $Z$. Ceci impliquerait que
 la torsion du groupe $H^3(Z,\Z)$, 
est nulle, ce qui
d'apr\`es Artin et Mumford n'est pas le cas.
Ainsi une fibre tr\`es g\'en\'erale  de $f$ n'est pas stablement rationnelle.

  Dans le pr\'esent article, nous montrons qu'une hypersurface quartique tr\`es g\'en\'erale   dans $\P^4_{\C}$
  n'est pas stablement rationnelle. Pour ce faire, nous rel\^achons les hypoth\`eses dans la m\'ethode de C. Voisin.
  D'une part nous autorisons l'espace total $X$ \`a ne pas \^{e}tre lisse, d'autre part nous
  rel\^{a}chons l'hypoth\`ese sur le diviseur exceptionnel d'une r\'esolution des singularit\'es $Z \to Y$.
  Que l'on puisse un peu rel\^acher cette derni\`ere hypoth\`ese est d\'ej\`a mentionn\'e dans  \cite[Remarque 1.2]{voisin}.

Nous donnons deux versions assez diff\'erentes de l'argument de sp\'ecialisation, l'un purement en termes de groupes
de Chow des z\'ero-cycles (\S 1), l'autre, essentiellement celui de Claire Voisin \cite{voisin},  en termes de correspondances (\S 2).
Un point essentiel de notre d\'emonstration utilise  l'homomorphisme de sp\'ecialisation de Fulton,
qui existe sous des hypoth\`eses tr\`es larges.

  Nous exhibons une hypersurface quartique singuli\`ere  $Y$ birationnelle \`a un solide d'Artin-Mumford,
dont nous construisons une r\'esolution des singularit\'es $Z \to Y$. Nous avons rel\'egu\'e cette construction 
\`a l'appendice A. 
Nous montrons que le diviseur exceptionnel remplit les conditions suffisantes
d\'egag\'ees aux paragraphes pr\'ec\'edents pour faire fonctionner la m\'ethode de sp\'ecialisation.

Le r\'esultat de sp\'ecialisation de z\'ero-cycles (\S 1)  montre   qu'une d\'eformation  g\'en\'erique
 de cette hypersurface quartique $Z$  n'est pas g\'eom\'etriquement stablement rationnelle,
 ni m\^eme r\'etracte rationnelle.

 Le  point de vue des correspondances (\S 2) \'etablit
 la non rationalit\'e stable pour les hypersurfaces quartiques ``tr\`es g\'en\'erales'' sur le corps des complexes.

Le point de vue ``groupe de Chow de z\'ero-cycles'' (\S 1)  \'etablit l'existence d'hypersurfaces quartiques
non stablement rationnelles d\'efinies sur une  cl\^oture alg\'ebrique de $\Q(t)$, et  montre  que
les param\`etres de telles hypersurfaces sont denses pour la topologie de Zariski sur l'espace projectif
param\'etrant ces vari\'et\'es (Th\'eor\`eme \ref{quartiques1}).
  En utilisant des sp\'ecialisations sur un corps fini, on peut m\^eme
  comme nous l'a obligeamment indiqu\'e O.~Wittenberg,   \'etablir  l'existence de telles 
hypersurfaces d\'efinies sur la cl\^oture alg\'ebrique de $\Q$ (Th\'eor\`eme \ref{reducmodp}).

Un exemple de quartique singuli\`ere  $Y$ dans $\P^4_{\C}$ avec une r\'esolution des singularit\'es
$Z \to Y$ satisfaisant les deux conditions : la torsion de 
$H^4(Z,\Z)$ est non nulle, et le diviseur
exceptionnel $E$ satisfait les conditions 
suffisantes mentionn\'ees
ci-dessus, avait d\'ej\`a \'et\'e construit par
J. Huh \cite{huh}. Pour l'exemple que nous construisons,
 point  n'est  
besoin de calculer  la torsion de 
$H^4(Z,\Z)$  : il suffit de renvoyer \`a l'article
d'Artin et Mumford, ou au calcul birationnel du groupe de Brauer de $Z$ \cite[Exemple 2.5]{CTOj}.

Le formalisme du \S 1 permet aussi d'\'etablir la non rationalit\'e stable, sur leur corps
de d\'efinition, de certaines vari\'et\'es. Pour $k$ un corps $p$-adique, ou un corps de nombres,
nous   montrons ainsi l'existence d'hypersurfaces
cubiques lisses de dimension 3 d\'efinies sur $k$ 
et qui ne sont pas stablement $k$-rationnelles (Th\'eor\`eme \ref{padiquestable}).

\bigskip

Soit $k$ un corps.
Une $k$-vari\'et\'e est un $k$-sch\'ema s\'epar\'e de type fini. Une $k$-vari\'et\'e int\`egre
est dite $k$-rationnelle si elle est $k$-birationnelle \`a un espace projectif $\P^n_{k}$.
Une $k$-vari\'et\'e int\`egre $X$ est dite stablement $k$-rationnelle s'il existe des espaces projectifs
$\P^n_{k}$ et $\P^m_{k}$ tels que $X \times_{k}\P^n_{k}$ est $k$-birationnel \`a $\P^m_{k}$.
Une $k$-vari\'et\'e int\`egre $X$ est dite r\'etracte rationnelle 
s'il existe des ouverts de Zariski non vides $U \subset X$ et $V \subset \P^m_{k}$ ($m$ convenable),
 et des $k$-morphismes $ f : U \to V$ et $g : V \to U$ tels que le compos\'e $g \circ f$ est l'identit\'e de $U$.
Une $k$-vari\'et\'e int\`egre stablement $k$-rationnelle est r\'etracte rationnelle.

Soit $X$ une $k$-vari\'et\'e projective  int\`egre. On dit qu'un $k$-morphisme 
$Z \to X$ est une d\'esingularisation de $X$ si $Z$ est une $k$-vari\'et\'e projective
lisse int\`egre et le morphisme $f$ est $k$-birationnel, c'est-\`a-dire qu'il induit un
isomorphisme $k(X) \stackrel{\sim}{\to} k(Z)$.

\section{Groupe de Chow des z\'ero-cycles et sp\'ecialisations}\label{algebrique}

 Soit $k$ un corps.

\defi{
On dit qu'un $k$-morphisme propre $f : X \to Y$ de $k$-vari\'et\'es est
 universellement
  $CH_{0}$-trivial si, pour tout corps $F$
contenant $k$, l'application induite $f_{*}: CH_{0}(X_{F}) \to CH_{0}(Y_{F})$  sur les groupes de Chow de z\'ero-cycles 
est un isomorphisme.}

\bigskip

Dans le cas particulier du morphisme structural d'une $k$-vari\'et\'e, on a la d\'efinition suivante.

\defi{On dit qu'une $k$-vari\'et\'e propre $X$ est 
  universellement
$CH_{0}$-triviale si son groupe de Chow de degr\'e z\'ero
est universellement \'egal \`a $\Z$, c'est-\`a-dire si,
 pour tout corps $F$
contenant $k$,  l'application degr\'e $deg_{F} : CH_{0}(X_{F}) \to \Z$ est un isomorphisme.}

\medskip

 Une telle $k$-vari\'et\'e $X$   est g\'eom\'etriquement connexe et  poss\`ede un z\'ero-cycle de degr\'e 1 sur le corps $k$.

\begin{ex} Soient $X_{i} \subset \P^n_{k}$, $i=1,2$, deux $k$-vari\'et\'es ferm\'ees universellement $CH_{0}$-triviales.
Si $X_{1} \cap X_{2}$ contient un point rationnel ou plus g\'en\'eralement un z\'ero-cycle de degr\'e 1,
alors la $k$-vari\'et\'e $X:=X_{1}\cup X_{2}$ est universellement $CH_{0}$-triviale.
\end{ex}

\begin{prop}\label{equivalencedeuxpointsdevue}
Soit   $X$ une $k$-vari\'et\'e propre,   lisse, g\'eom\'etriquement int\`egre de dimension $n$.

Soit $K$ le corps des fonctions rationnelles de $X$.
Les conditions suivantes sont \'equivalentes :

 (i) La $k$-vari\'et\'e   $X$ est  universellement $CH_{0}$-triviale.

(ii) La $k$-vari\'et\'e $X$ poss\`ede un     z\'ero-cycle  de degr\'e 1,
et  la fl\`eche $deg_{K} : CH_{0}(X_{K}) \to \Z$ est un isomorphisme.

(iii)  Il existe une sous-vari\'et\'e ferm\'ee $D \subset X$ de codimension 1,
   un  z\'ero-cycle $z_{0}$ de degr\'e 1 sur $X$
 et
 un cycle $Z \in Z_{n}(X\times X)$ support\'e dans $D \times X$, tels que le cycle
 $$\Delta_{X} - Z - X\times z_{0} \in Z_{n}(X\times X)$$
  ait une classe nulle dans $CH_{n}(X\times X)$.

  Si ces conditions sont satisfaites, 
  la propri\'et\'e (iii) vaut en y rempla\c cant
 $z_{0}$ par tout autre z\'ero-cycle de degr\'e 1.

   \end{prop}

   \begin{proof}
 Voir \cite[Lemma 1.3]{ACTP}),
   qui utilise la th\'eorie des correspondances sur les vari\'et\'es
 propres et lisses sur un corps  \cite[Chap. 16]{fulton}.
\end{proof}

Dans la situation du point (iii) de la proposition, on dit que l'on a {\it une d\'ecomposition de Chow de la diagonale de la $k$-vari\'et\'e $X$}.

\begin{lem}\label{retractCh0trivial}
Soit $X$ une $k$-vari\'et\'e int\`egre projective et lisse. Si la $k$-vari\'et\'e $X$
est r\'etracte rationnelle, c'est une $k$-vari\'et\'e
 universellement
 $CH_{0}$-triviale.
\end{lem}

\begin{proof}  
Par hypoth\`ese, il existe
$U \subset X$ et $V \subset \P^n_{k}$ des ouverts non vides, et $f : U \to V$ et $g : V \to U$
des $k$-morphismes tels que $g \circ f$ soit l'identit\'e de $U$.

Supposons que $V$ poss\`ede un $k$-point $A$, ce qui est certainement le cas si $k$
est infini. Soit $Q$ un point ferm\'e de $V$ tel que l'inclusion naturelle  de corps r\'esiduels $k(g(Q)) \subset k(Q)$ soit
un isomorphisme.
Notons $F= k(Q)$.
Comme $F$ est un quotient de $F \otimes_{k}F$, 
il existe un point $F$-rationnel  $R$ de $V_{F}$ d'image $Q$ par la projection $V_{F} \to V$. 
il existe
une $F$-droite projective  $L \subset \P^n_{F}$ passant par  les points $F$-rationnels $A_{F}$ et $R$.
L'application $g_{F}$ induit une $F$-application rationnelle de la droite $L$ vers $U_{F} \subset X_{F}$
et donc un $F$-morphisme de $L$ vers la $F$-vari\'et\'e propre  $ X_{F}$.  
Ainsi  le z\'ero-cycle $g_{F}(R) - g_{F}(A_{F})$ est rationnellement \'equivalent \`a z\'ero sur $X_{F}$,
et donc  le z\'ero-cycle
$g_{*}(Q) - [F:k] g(A)$, qui est son image par la projection $X_{F} \to X$,
 est rationnellement \'equivalent \`a z\'ero sur $X$.

Soit $z$ un z\'ero-cycle de degr\'e z\'ero sur $X$. Par un lemme de d\'eplacement facile (\cite[Compl\'ement, page 599]{CTmoving} pour $k$ parfait infini,
\cite[Cor. 6.7]{GLL} pour $k$ quelconque),
ce z\'ero-cycle est rationnellement \'equivalent sur $X$ \`a un z\'ero-cycle $z_{1}= \sum_{P}n_{p}P$
de degr\'e z\'ero \`a support dans $U$.

Comme $g\circ f$ est l'identit\'e, les points ferm\'es $P$ et $f(P)$ ont des corps r\'esiduels isomorphes.
Par l'argument pr\'ec\'edent, chaque z\'ero-cycle $$P- [k(P):k]g(A) = g_{*}f_{*}(P) - [k(P):k]g(A)$$
est rationnellement \'equivalent \`a z\'ero sur $X$. Donc
le z\'ero-cycle de degr\'e z\'ero $z_{1} =\sum_{P}n_{p}P$ est rationnellement \'equivalent \`a z\'ero
sur $X$. Il en est donc de m\^eme de $z$ sur $X$.

Le cas d'un corps fini se traite par un argument de corestriction-restriction,
 par extension \`a des extensions finies de degr\'e premi\`eres
entre elles sur lesquelles $V$ poss\`ede un point rationnel.
\end{proof}

\begin{rema}
Comme nous l'avait indiqu\'e A. Merkurjev, ce lemme est aussi  une cons\'equence du fait, d\^{u} \`a M. Rost, que $CH_{0}$ s'\'etend \`a la cat\'egorie 
des correspondances rationnelles. La d\'emonstration \cite[Cor. RC.12]{KM} n'utilise pas
la r\'esolution des singularit\'es.  
\end{rema}

\begin{lem}\label{corpsalgclos}
Soient   $X$ et $Y$ deux $k$-vari\'et\'es g\'eom\'etriquement int\`egres.
S'il existe un corps $L$ contenant $k$ tel que l'une des propri\'et\'es suivantes
est satisfaite :

(i) $X_{L}$ est $L$-birationnelle \`a $Y_{L}$,

(ii) $X_{L}$ est $L$-rationnelle,

(iii) $X_{L}$ est stablement $L$-rationnelle,

(iv) $X_{L}$ est une $L$-vari\'et\'e r\'etracte rationnelle,

\noindent alors   cette propri\'et\'e vaut pour une extension $L$ finie convenable de $k$.
\end{lem}

\begin{proof}
Montrons (i).
 On se ram\`ene au cas $k$ alg\'ebriquement clos
et $L=k(Z)$ est le corps des fonctions d'une $k$-vari\'et\'e int\`egre. Les $k$-vari\'et\'es
$X \times_{k}Z $ et $Y \times_{k}Z$ sont birationnellement \'equivalentes
par une \'equivalence qui respecte la projection sur $Z$. Il existe des ouverts
non vides $U \subset X \times_{k}Z $ et $V \subset Y \times_{k}Z$
qui sont $k$-isomorphes. 
Il existe une extension finie $F$ de $k$ et 
un $F$-point de $Z$
tel que les fibres au-dessus de ce point soient des ouverts non vides
de $X$, resp. de $Y$, et qui soient isomorphes. Ceci \'etablit  l'\'enonc\'e dans le  cas (i), 
lequel implique  imm\'ediatement
l'\'enonc\'e dans
les cas  (ii) et (iii).
La d\'emonstration dans le cas  (iv) est analogue.
    \end{proof}

\begin{prop}\label{chowisolissesing}
Soit $f : Z \to Y$ un $k$-morphisme propre de $k$-vari\'et\'es alg\'ebriques. Les hypoth\`eses 
  suivantes  sont \'equivalentes :

(i) Pour tout point $M$ du sch\'ema $Y$, de corps r\'esiduel $\kappa(M)$,  la fibre $Z_{M}$ est une
$\kappa(M)$-vari\'et\'e
 universellement
 $CH_{0}$-triviale.

(ii) Pour tout corps $F$ contenant $k$ et tout point $M \in Y(F)$, la fibre $Z_{M}$ est une $F$-vari\'et\'e
$CH_{0}$-triviale.

Elles   impliquent  que  le $k$-morphisme $f$  est 
  universellement
$CH_{0}$-trivial.

\end{prop}
\begin{proof}
L'\'equivalence des deux hypoth\`eses est imm\'ediate.
Pour \'etablir l'\'enonc\'e, il suffit de montrer que pour $F=k$, la fl\`eche $f_{*} : CH_{0}(Z) \to CH_{0}(Y)$
est un isomorphisme.
L'hypoth\`ese assure imm\'ediatement que la fl\`eche $f_{*} : CH_{0}(Z) \to CH_{0}(Y)$ est surjective.
Soit $z $ un z\'ero-cycle sur $Z$. Si $f_{*}(z)$ est rationnellement \'equivalent \`a z\'ero sur $Y$, alors
il existe des courbes ferm\'ees int\`egres $C_{i} \subset Y$ et des fonctions rationnelles $g_{i} \in k(C_{i})$
telles que $f_{*}(z) = \sum_{i} div_{C_{i}}(g_{i})$. En appliquant l'hypoth\`ese  au corps des fonctions 
des courbes $C_{i}$,  
on trouve des courbes ferm\'ees  int\`egres $D_{i}^j \subset Z$ en nombre fini telles que  $f$ induise des
 morphismes finis surjectifs $f_{i}^j : D_{i}^j  \to C_{i}$ tels que, pour chaque $i$,
 on ait une \'egalit\'e $\sum_{j}n_{i}^j deg(f_{i}^j)=1,$ avec les $n_{i}^j \in \Z$.
 On note encore $g_{i}$ la fonction rationnelle sur $D_{i}^j$ image r\'eciproque
 par  $f_{i}^j $ de la fonction rationnelle $g_{i}$ sur $C_{i}$.
  Le z\'ero-cycle $z' :=z - \sum_{i} [\sum_{j} n_{i}^j div_{D_{i}^j}(g_{i})]$ sur $Z$  satisfait
 $f_{*}(z') =0 $ comme z\'ero-cycle sur $Y$.
  Il existe donc des points ferm\'es $Q_{j}$ de $Y$
en nombre fini tels que le z\'ero-cycle $z'$ soit sur $Z$ somme de z\'ero-cycles $z_{j} $,
chaque $z_{j}$ \'etant support\'e sur la fibre $Z_{Q_{j}}=f^{-1}(Q_{j})$. L'hypoth\`ese
assure que chacun des $z_{j}$ est rationnellement \'equivalent \`a z\'ero sur $Z_{Q_{j}}$, donc sur $Z$.
Ainsi $z$ est rationnellement \'equivalent \`a z\'ero sur $X$.
La fl\`eche $f_{*} : CH_{0}(Z) \to CH_{0}(Y)$ est donc injective.\end{proof}

\begin{prop}\label{chowunivtrivial}
Soit $f : Z \to Y$ un $k$-morphisme propre birationnel  de $k$-vari\'et\'es alg\'ebriques
projectives  g\'eom\'etriquement int\`egres. Supposons :

  (i) La $k$-vari\'et\'e $Z$ est lisse et poss\`ede un z\'ero-cycle de degr\'e 1.
  
  (ii)  Le $k$-morphisme $f $ est 
  universellement
  $CH_{0}$-trivial.

(iii)
 Il existe un ouvert non vide $U \subset Y$ lisse sur $k$, 
  d'image r\'eciproque $V=f^{-1}(U) \subset Z$
  tel que $f : V \to U$ soit un isomorphisme, et tel que pour tout corps $F$ contenant $k$,  
 tout z\'ero-cycle de degr\'e z\'ero   \`a support dans $U_{F}$ 
     est rationnellement \'equivalent \`a z\'ero sur $Y_{F}$.

  Alors la $k$-vari\'et\'e $Z$ est 
  universellement
  $CH_{0}$-triviale.
  \end{prop}

\begin{proof}
Les hypoth\`eses \'etant invariantes par changement de corps $k \subset F$, il
 suffit d'\'etablir que 
sous les hypoth\`eses ci-dessus
  la fl\`eche  $deg_{k} : CH_{0}(Z) \to \Z$ est un isomorphisme.
  D'apr\`es (i), la fl\`eche est surjective.
 Un lemme de d\'eplacement facile et classique assure que
 tout z\'ero-cycle de degr\'e z\'ero sur  la $k$-vari\'et\'e lisse $Z$
est rationnellement \'equivalent  
\`a un z\'ero-cycle de degr\'e z\'ero dont le support est dans $V$.
D'apr\`es (iii), l'image  du cycle $f_{*}(z)$ dans   $CH_{0}(Y)$ est nulle.
L'hypoth\`ese~(ii)   assure que $f_{*} : CH_{0}(Z) \to CH_{0}(Y)$
est un isomorphisme. Comme cet isomorphisme respecte le degr\'e, on conclut que la fl\`eche
 $deg_{k} : CH_{0}(Z) \to \Z$ est un isomorphisme.  
\end{proof}

\begin{prop}\label{specialisation}
Soit $A$ un anneau de valuation discr\`ete de corps des fractions $K$
et de corps r\'esiduel $k$.
Soit ${\mathcal X}$ un $A$-sch\'ema propre et plat, $X={\mathcal X}\times_{A}K$
la fibre g\'en\'erique 
et $Y={\mathcal X}\times_{A}k$ la fibre sp\'eciale.

(i) On a une application naturelle de sp\'ecialisation
$ CH_{0}(X) \to CH_{0}(Y)$; elle est compatible avec
les applications degr\'e \`a valeurs dans $\Z$.

(ii) Si $A$ est hens\'elien et $X/K$ g\'eom\'etriquement int\`egre
 poss\`ede une d\'esingu\-la\-ri\-sa\-tion $p : \tilde{X} \to X$
telle que la fl\`eche $deg_{K} : CH_{0}(\tilde{X}) \to \Z$ est un isomorphisme,
alors tout z\'ero-cycle de degr\'e z\'ero  de $Y$ \`a support dans le lieu lisse $Y_{lisse}$ de $Y$
  a une   classe  nulle dans $CH_{0}(Y)$.

\end{prop}

\begin{proof}
 L'\'enonc\'e (i) est un cas particulier de la construction
d'homomorphismes
de sp\'ecialisation
(Fulton  \cite[Prop. 2.6]{fulton}, \cite[\S 4]{fulton75}).

Montrons (ii). 
Soit $X_{lisse} \subset  X$ 
l'ouvert de lissit\'e.  Soit $U \subset  X_{lisse}$ un ouvert non vide 
tel que la fl\`eche induite  $p : p^{-1}(U) \to U$ soit un isomorphisme.
  Par Hensel, un z\'ero-cycle $z$ de degr\'e z\'ero sur $Y_{lisse}$ se rel\`eve en
   un  z\'ero-cycle $z_{1}$, de degr\'e z\'ero, support\'e sur $X_{lisse}$.
   Un lemme de d\'eplacement facile  assure que le z\'ero-cycle
   $z_{1}$ est rationnellement \'equivalent dans $X$ \`a un z\'ero-cycle
   $z_{2}$, de degr\'e z\'ero,  dont le support est dans $U$. Le z\'ero-cycle $z_{2}  $ 
   est l'image par $p$
  d'un z\'ero-cycle de degr\'e z\'ero $z_{3}$ sur $\tilde{X}$.
  L'application compos\'ee $CH_{0}(\tilde{X}) \to CH_{0}(X) \to CH_{0}(Y)$
  envoie la classe de $z_{3}$ sur la classe de $z$. Or, par hypoth\`ese, $z_{3}=0 \in CH_{0}(\tilde{X})$.
   \end{proof}

 \begin{rema}
 Pour la construction de l'homomorphisme de sp\'e\-cia\-lisation, nous avons cit\'e
  \cite[ \S 2, Prop. 2.6]{fulton}) et non la r\'ef\'erence plus \'evidente
  \cite[\S 20.3]{fulton}). La raison est que la d\'emonstration donn\'ee dans
    \cite[\S 20.3]{fulton}) passe par \cite[Theorem 6.3]{fulton}  et donc par  la d\'eformation au c\^one normal
\cite[ \S 5]{fulton},  r\'esultat \'etabli pour les vari\'et\'es au-dessus d'un corps.
Or nous ne voulons pas nous limiter au cas o\`u $A$ est un anneau local d'une courbe lisse sur un corps.
  \end{rema}

On renvoie \`a \cite{rost} pour la th\'eorie des modules de cycles de Rost (voir aussi les rappels
dans \cite{merkurjev} et \cite{ACTP}).
Rappelons que pour $k$ un corps et $n$ un entier inversible dans ce corps,
 les groupes de cohomologie \'etale $H^{i}(\bullet, \mu_{n}^{\otimes j})$   d\'efinissent une th\'eorie
des modules de cycles sur les $k$-vari\'et\'es, et que pour $Y$ une $k$-vari\'et\'e projective int\`egre de
d\'esingularisation $Z \to Y$, le groupe de cohomologie non ramifi\'ee  $H^{2}_{nr}(k(Y)/k,\mu_{n})$
s'identifie au sous-groupe de $n$-torsion du groupe de Brauer $\Br(Z)$.

\begin{theo}\label{principal}
Soit $A$ un anneau de valuation discr\`ete   de corps des fractions $K$
et de corps r\'esiduel $k$.
Soit ${\mathcal X}$ un $A$-sch\'ema fid\`element plat et propre sur $A$
\`a fibres g\'eom\'etriquement int\`egres.

Supposons  que
la fibre sp\'eciale $Y={\mathcal X}\times_{A}k$  
  poss\`ede  une d\'esingularisation  $f : Z \to Y$,
   avec  $Z$ lisse sur $k$,
telle que le morphisme $f$  est 
 universellement
$CH_{0}$-trivial, et que $Z$ poss\`ede un
z\'ero-cycle de degr\'e 1.

Supposons que la  fibre g\'en\'erique $X={\mathcal X}\times_{A}K$  
admet une r\'esolution des singularit\'es $\tilde{X} \to X$,   avec  $\tilde{X}$ lisse sur $K$.

Chacun des \'enonc\'es (i), (ii), (iii) ci-dessous implique le suivant   :

(i) La $K$-vari\'et\'e $X$ est r\'etracte rationnelle.

(ii) La $K$-vari\'et\'e $\tilde{X}$ est 
universellement
$CH_{0}$-triviale.

 (iii) La $k$-vari\'et\'e $Z$ est 
 universellement
$CH_{0}$-triviale.

Cette derni\`ere propri\'et\'e implique :

(a) Pour tout module de cycles    $M^{i}$  sur le corps $k$,
pour tout corps~$L$ contenant 
$k$,
et tout $i\geq 0$, la fl\`eche
   $M^{i}(L) \to M^{i}_{nr}(L(Z)/L)$ est un isomorphisme.
   
   (b) Pour tout corps $L$ contenant $k$, la fl\`eche naturelle
   $\Br(L) \to \Br(Z_{L})$ est un isomorphisme.
 \end{theo}

\begin{proof} 
Supposons (i). Le lemme   \ref{retractCh0trivial} donne alors (ii).
Supposons (ii). Comme la fibre  $Y$ est g\'eom\'etriquement int\`egre,
son point g\'en\'erique $\eta$  est r\'egulier sur $\mathcal{X}$, l'anneau local $O_{{\mathcal X},\eta}$
de $\mathcal{X}$ en ce point g\'en\'erique est un anneau de valuation
discr\`ete de corps des fractions le corps  des fonctions $K(X)$ de $X$
et de corps r\'esiduel le corps des fonctions $k(Y)$ de $Y$.
Notons $B$ le compl\'et\'e de l'anneau de valuation discr\`ete  $O_{{\mathcal X},\eta}$.
Soit $F$ le corps des fractions de $B$. Le corps r\'esiduel de $B$ est $k(Y)=k(Z)$.
La fl\`eche naturelle $A \to B$ est un homomorphisme local,
induisant $k \to k(Y)$ sur les corps r\'esiduels.
On consid\`ere le $B$-sch\'ema ${\mathcal X}\times_{A}B$. Sa fibre
g\'en\'erique est $X\times_{K}F$, qui admet la d\'esingularisation
$\tilde{X}\times_{K}F \to X\times_{K}F$.
Sa fibre sp\'eciale  est $Y\times_{k}k(Y)$, qui 
 admet la d\'esingularisation $Z_{k(Y)} \to Y_{k(Y)}$.
 Le $k(Y)$-morphisme $Z_{k(Y)} \to Y_{k(Y)}$
 est universellement $CH_{0}$-trivial.

L'hypoth\`ese (ii) assure que le degr\'e 
$deg_{F} : CH_{0}({\tilde X}_{F}) \to \Z$ est un isomorphisme.
Une application des propositions \ref{chowunivtrivial} et \ref{specialisation}
au $B$-sch\'ema ${\mathcal X}\times_{A}B$  montre alors
que la fl\`eche degr\'e $CH_{0}(Z_{k(Z)}) \to \Z$ est un isomorphisme.
La proposition \ref{equivalencedeuxpointsdevue} appliqu\'ee \`a la $k$-vari\'et\'e $Z$
assure alors que cette vari\'et\'e est universellement $CH_{0}$-triviale.
\end{proof}

\begin{rema}
On peut se dispenser de passer 
par l'anneau $O_{{\mathcal X},\eta}$. De fait,  pour  tout anneau de valuation discr\`ete $A$ de corps r\'esiduel $k$
et tout corps $F$ contenant $k$,  il existe un anneau de valuation discr\`ete $B$ de corps r\'esiduel $F$
et un homomorphisme local de $A$ dans $B$ induisant l'inclusion  $k \subset F$ (J-P. Serre nous signale les \'enonc\'es   g\'en\'eraux
de Bourbaki sur le
``gonflement des anneaux locaux'' \cite[Chap. IX, Appendice, \S 2, Corollaire du Th\'eor\`eme 1, et Exercice 4]{AC}.)
 \end{rema}

\begin{theo}\label{principalgeom}
Soit $A$ un anneau de valuation discr\`ete 
de corps des fractions $K$
et de corps r\'esiduel $k$ alg\'ebriquement clos.
Soit ${\mathcal X}$ un $A$-sch\'ema fid\`element plat et propre sur $A$
\`a fibres g\'eom\'etriquement int\`egres.

Supposons  que
la fibre sp\'eciale $Y={\mathcal X}\times_{A}k$  
 poss\`ede  une d\'esingularisation  $f : Z \to Y$
telle que le morphisme $f$  est 
 universellement
$CH_{0}$-trivial.

Soit $\overline{K}$ une cl\^oture alg\'ebrique de $K$.
Supposons que la  fibre g\'en\'erique g\'eom\'etrique $\overline{X}:={\mathcal X}\times_{A}{\overline K}$  
admet une d\'esingularisation $\tilde{X} \to \overline{X}$.

Chacun des \'enonc\'es (i), (ii), (iii) ci-dessous implique le suivant :

(i) La $\overline{K}$-vari\'et\'e $\overline{X}$ est r\'etracte rationnelle.

(ii) La  $\overline{K}$-vari\'et\'e  $\tilde{X}$  est 
universellement
$CH_{0}$-triviale.

 (iii) La $k$-vari\'et\'e $Z$ est 
universellement
$CH_{0}$-triviale.

Cette derni\`ere propri\'et\'e implique :

(a) Pour tout module de cycles    $M^{i}$  sur le corps $k$,
pour tout corps~$L$ contenant 
$k$,
et tout $i\geq 0$, la fl\`eche
   $M^{i}(L) \to M^{i}_{nr}(L(Z)/L)$ est un isomorphisme.
   
   (b) Pour tout corps $L$ contenant $k$, la fl\`eche naturelle
   $\Br(L) \to \Br(Z_{L})$ est un isomorphisme.
   
   (c) $\Br(Z)=0$.
\end{theo}

\begin{proof}

Supposons (i). Le lemme   \ref{retractCh0trivial} donne alors (ii).
Supposons (ii).  Quitte \`a remplacer l'anneau  de valuation discr\`ete $A$ par son compl\'et\'e,
on peut supposer $A$ complet.
Il existe une sous-extension finie   $K \subset L \subset {\overline K}$,
une $L$-vari\'et\'e propre et lisse  $W$ avec $W\times_{L}{\overline K}=  \tilde{X}$
munie d'un $L$-morphisme
$W \to X_{L}$ qui par changement de base de $L$ \`a $\overline K$
s'identifie \`a  $\tilde{X} \to \overline{X}$.  La proposition \ref{equivalencedeuxpointsdevue}
montre que quitte \`a remplacer $L$ par une extension finie, on peut supposer
que   la $L$-vari\'et\'e $W$ est universellement $CH_{0}$-triviale.
 Comme $A$ est complet, la fermeture int\'egrale $B$
de $A$ dans $L$ est un anneau de valuation discr\`ete complet (\cite[Chap. II, \S 2, Prop. 3]{serre}).
Son corps r\'esiduel
est $k$.
 Une application du th\'eor\`eme 
\ref{principal} au $B$-sch\'ema
$\mathcal{X}\times_{A}B$ termine la d\'emonstration.
  \end{proof}

\begin{rema}
Ce th\'eor\`eme \'etend le  th\'eor\`eme
\cite[Thm. 1.1 (i)]{voisin} de C.~Voisin, qui porte sur le cas 
 $k= \C$, $\mathcal{X}$   sch\'ema  r\'egulier
et $Y$ \`a singularit\'es quadratiques ordinaires.
\end{rema}

\begin{rema}
Dans la d\'emonstration, m\^eme si $\mathcal{X}$ est r\'egulier, 
le changement de base $A \to B$ peut donner naissance \`a un sch\'ema
$\mathcal{X}\times_{A}B$ qui n'est pas r\'egulier. Mais la fibre
sp\'eciale ne change pas.
\end{rema}

\begin{theo}\label{quartiques1}
Dans l'espace projectif $\P^N$ param\'etrant les hypersurfaces quartiques dans $\P^4_{\C}$,
l'ensemble $E$  des points de $\P^N(\C)$  param\'etrant une quartique non r\'etracte rationnelle, et donc en particulier
non stablement rationnelle,  est
Zariski-dense.  Fixons un plongement ${\overline{\Q}}(t) \subset \C$.
Le sous-ensemble de $E$ form\'e des points dont les coordonn\'ees sont dans ${\overline{\Q}}(t)$ est   Zariski dense
dans $\P^N$.

\end{theo}
\begin{proof}
Soit $W \subset \P^N_{\Q}$ le ferm\'e correspondant aux quartiques singuli\`eres.
D'apr\`es l'appendice A, ou d'apr\`es
  J. Huh \cite{huh}, il existe une quartique singuli\`ere $Y$ d\'efinie sur $\overline{\Q}$
et  une r\'esolution des singularit\'es $f : Z \to Y$ telles que le $\overline{\Q}$-morphisme  $f$ est
universellement
$CH_{0}$-trivial
et  que $\Br (Z_{\C}) \neq 0$, ce qui implique
 $\Br (Z) \neq 0$.
Soit $D=\P^1_{\overline{\Q}} \subset \P^N_{\overline{\Q}}$ une droite passant par un 
$\overline{\Q}$-point $M\in \P^N$
associ\'e \`a la quartique $Y$, et non contenue dans $W_{\overline{\Q}}$. 
Soit $A$ l'anneau local de la droite  $D$ au point $M$,
et soit $K$ son corps des fonctions.
Le th\'eor\`eme \ref{principalgeom}
implique que la quartique lisse sur $K=\overline{\Q}(t)$ correspondant au point
g\'en\'erique de $D$ n'est pas g\'eom\'etriquement r\'etracte rationnelle.

Tout choix d'un point   $R$ de $D(\C) \setminus D(\overline{\Q})$ d\'efinit un plongement
$\overline{\Q}(D) \hookrightarrow \C$, qui d\'efinit un plongement d'une cl\^oture alg\'ebrique
de $\overline{\Q}(D) $ dans $\C$. Par le lemme  \ref{corpsalgclos},
la quartique lisse associ\'ee au point $R \in
 D(\C) \subset \P^N(\C)$ n'est donc pas r\'etracte rationnelle.

 On voit donc que, pour  tout point  $R$ de $\P^N(\C)$ non dans $\P^N(\overline{\Q})$
 et non dans $W$
 tel que la droite joignant $R$ et $M$ soit d\'efinie sur $\overline{\Q}$, la quartique lisse associ\'ee \`a $R$
 n'est pas r\'etracte  rationnelle.
\end{proof}

\begin{rema}
 Soit  $Y$ un rev\^etement double d'une surface  quartique d'Artin-Mumford \cite{artinmumford}
d\'efinie sur $\overline{\Q}$. La vari\'et\'e $Y$ a exactement 10 points singuliers quadratiques ordinaires.
Une  d\'esingularisation $Z \to Y$ s'obtient par  \'eclatement de ces 10 points, au-dessus de
chacun de ces points on a une  quadrique lisse de dimension 2.
La proposition \ref{chowisolissesing}  montre imm\'ediatement que le morphisme de d\'esingularisation
est   
 universellement
$CH_{0}$-trivial.

On retrouve ici
le th\'eor\`eme de
 C. Voisin \cite[Cor. 1.4]{voisin} :
 {\it Dans l'espace projectif $\P^N$ param\'etrant les surfaces quartiques  $f=0$ dans $\P^3_{\C}$,
l'ensemble des points de $\P^N(\C)$  tels que le rev\^etement double associ\'e $z^2-f=0$
dans l'espace  multihomog\`ene $\P(2,1,1,1,1)$ ne soit pas une vari\'et\'e stablement rationnelle
est  Zariski-dense.}

 \end{rema}

\begin{rema}
La d\'emonstration du th\'eor\`eme \ref{quartiques1} que nous avons donn\'ee repose sur le th\'eor\`eme \ref{principalgeom},
donc  sur le th\'eor\`eme \ref{principal}, donc
sur l'op\'eration de sp\'ecialisation 
des z\'ero-cycles $CH_{0}(X) \to CH_{0}(Y)$, o\`u $X$,
  resp. $Y$, est 
 la fibre g\'en\'erique, resp. la fibre sp\'eciale,
d'un $A$-sch\'ema propre fid\`element plat sur un anneau de valuation discr\`ete $A$ de corps des fractions $K$
et de corps r\'esiduel $k$ (Proposition \ref{specialisation}).

On pourrait si l'on voulait ignorer cette op\'eration en la rempla\c cant par l'op\'eration
de sp\'ecialisation de la $R$-\'equivalence $X(K)/R \to Y(k)/R$, facile \`a d\'efinir \cite{madore}. 
Pour  $K$ de caract\'eristique z\'ero,  cette m\'ethode suffit  
\`a montrer les implications (i)  $\Longrightarrow$ (iii) dans les
 th\'eor\`emes \ref{principal} et  \ref{principalgeom}. 
  L'hypoth\`ese de caract\'eristique z\'ero
 est utilis\'ee pour   montrer, via le th\'eor\`eme d'Hironaka, que sous  l'hypoth\`ese (i) du
 th\'eor\`eme  \ref{principal},   la $R$-\'equivalence sur $\tilde{X}(K)$ est triviale (en utilisant \cite[Prop. 10]{CTS}).

 \end{rema}

Nous remercions O. Wittenberg de nous avoir sugg\'er\'e l'\'enonc\'e suivant.

\begin{theo}\label{reducmodp}
Il existe des hypersurfaces quartiques lisses $X \subset \P^4_{\C}$ d\'efinies sur la  cl\^oture alg\'ebrique 
$\overline{\Q}$ de $\Q$ dans $\C$, et qui ne sont pas universellement $CH_{0}$-triviales, en particulier
qui  ne sont pas   r\'etractes rationnelles.
\end{theo} 

\begin{proof}
Soit $Y \subset \P^4_{\overline{\Q}}$ une quartique singuli\`ere poss\'edant  
 une r\'esolution des singularit\'es $f : Z \to Y$  comme construite dans
  l'appendice A. Le morphisme $f$ est  universellement $CH_{0}$-trivial.
Le sous-groupe de  torsion  2-primaire $\Br(Z)\{2\}$ du groupe de Brauer $\Br(Z)$ est non nul. Dans la situation consid\'er\'ee,
o\`u $Z$ est une vari\'et\'e rationnellement connexe, donc satisfait $H^2(Z,O_{Z})=0$ et 
donc   $\rho=b_{2}$,  le groupe $\Br(Z)\{2\}$
 s'identifie
  au sous-groupe de torsion 2-primaire $H^3_{{\acute{e}}t}(Z, \Z_{2})\{2\}$  du
troisi\`eme groupe de cohomologie 2-adique (\cite[II, \S 3; III, \S 8] {grbr}).

D'apr\`es  la description des fibres de $f$ donn\'ee dans la
proposition \ref{lemodele} de l'appendice A,
il existe une extension finie $K$ de $\Q$ sur laquelle  $Z , Y$  et $Z \to Y$ sont
 d\'efinis et  sur laquelle $f : Z \to Y$ est un $K$-morphisme  universellement $CH_{0}$-trivial.
 
 Il existe un ouvert $U$  non vide du spectre  de l'anneau des entiers de $K$, des $U$-sch\'emas
 $\mathcal{Z}$, $\mathcal{Y}$ et un $U$-morphisme   $\mathcal{Z} \to \mathcal{Y}$
 \'etendant $Z \to Y$, de telle sorte que pour toute place $v \in U$,
  on obtienne
 par r\'eduction une situation analogue $Z_{\kappa(v)} \to Y_{\kappa(v)}$  sur le corps fini r\'esiduel $\kappa(v)$ :
 c'est une r\'esolution des singularit\'es de $Y_{\kappa(v)}$,  et la fl\`eche  $Z_{\kappa(v)} \to Y_{\kappa(v)}$  
 est universellement $CH_{0}$-triviale. On peut supposer que $U$ ne contient pas de place 2-adique.
Soit $S$ l'ensemble fini des nombres premiers 
qui sont caract\'eristiques r\'esiduelles du  
compl\'ementaire
de $U$ dans le spectre de l'anneau des
entiers de $k$.
  Notons $\overline{\kappa(v)}$ une cl\^oture alg\'ebrique de $\kappa(v)$.
 Le th\'eor\`eme de changement de base propre et lisse (\cite[Chap. V, Thm. 3.1]{deligne}) assure alors  
 $H^3_{{\acute{e}}t}(Z_{\overline{\kappa(v)}},     \Z_{2})   \{2\} \neq 0$ pour tout $v\in U$, et donc
 $\Br(   Z_{\overline{\kappa(v)}}    ) \neq 0$.

 Soit $  \P^N_{\Q}$ l'espace projectif param\'etrant les hypersurfaces quartiques.
 Il existe une hypersurface  $ \Sigma  \subset  \P^N_{\Q}$, d\'efinie par une
 forme homog\`ene non nulle $\Delta$ \`a coefficients entiers,
 telle que toute hypersurface quartique de param\`etre hors de $\Sigma$
 est lisse.  L'hypersurface $Y_{\kappa(v)}$ correspond \`a un point $m_{v} \in \P^N(\kappa(v))$.
 Soit $A$ l'anneau local de l'anneau des entiers de $K$ en $v$.
Il existe un point $m \in \P^N(A)$ se r\'eduisant sur $m_{v}$ et tel que $\Delta(m) \neq 0 \in K$.
On applique alors le th\'eor\`eme \ref{principalgeom}.
\end{proof}

\medskip

Clemens et Griffiths ont montr\'e qu'une hypersurface cubique lisse dans $\P^4_{\C}$
n'est jamais rationnelle. 
Sur un corps $k$ quelconque, on peut se poser la question de la $k$-rationalit\'e, stable
ou non, des hypersurfaces cubiques lisses dans $\P^n_{k}$, $n \geq 3$,
du moins pour celles qui poss\`edent un $k$-point.
Le cas $n=3$  a \'et\'e  \'etudi\'e (Shafarevich, Manin). Sur le corps $\R$
des r\'eels, une  hypersurface cubique   lisse $X$ telle que $X(\R)$ ait
deux composantes connexes  pour la topologie r\'eelle ne saurait \^etre
r\'etracte rationnelle.
 Pour $n\geq 4$, 
 la question de la rationalit\'e stable sur $k$ est ouverte lorsque
 $k$ est alg\'ebriquement clos, et lorsque $k$ est fini.
 Pour $k=\C((x))((y))$, D. Madore \cite{madore2} a montr\'e que le groupe
 de Chow $A_{0}(X) \subset CH_{0}(X)$ des cycles de dimension et de
 degr\'e z\'ero sur 
l'hypersurface $X \subset \P^4_{k}$ d\'efinie en coordonn\'ees homog\`enes par l'\'equation
  $$ T_{0}^3+T_{1}^3+x T_{2}^3+ y  T_{3}^3 +x y  T_{4}^3=0$$
est non nul. Ceci implique que $X$ 
n'est pas r\'etracte rationnelle. La d\'emonstration proc\`ede par
  sp\'ecialisation du groupe de Chow. Elle utilise un espace principal
  homog\`ene non trivial sous une vari\'et\'e ab\'elienne sur   $\C((x))$.
On ne peut donc y remplacer $\C((x))$ par un corps fini $\F_{p}$.
Ceci laissait donc ouverte la question analogue sur $k=\F_{p}((y))$
ou $k$ un corps $p$-adique.  
La m\'ethode de sp\'ecialisation du pr\'esent
article permet de r\'esoudre cette quesion.

\begin{theo}\label{padiquestable}
Sur tout corps $p$-adique $k$, il existe 
une hypersurface cubique 
$X \subset \P^4_{k}$  poss\'edant un point rationnel et qui
n'est pas universellement  $CH_{0}$-triviale, et qui n'est donc
  pas r\'etracte rationnelle.
\end{theo}

\begin{proof}
Soient    $k$ un corps $p$-adique,  $A \subset k$ son anneau des entiers, $\pi$ une uniformisante et $\F$ le  corps fini  r\'esiduel.
  Soit $K/k$ l'extension cubique non ramifi\'ee. Il existe un \'el\'ement $\alpha$ dans l'anneau des
  entiers de $K$ tel que $K=k(\alpha)$ et la classe $\beta$ de $\alpha$ dans le corps r\'esiduel  engendre l'extension cubique 
  $E$ de $\F$.
 Soit $\Phi \in \Z_{p}[u,v,w,x,y]$ une forme  telle que 
$\Phi =0$
d\'efinisse un $\Z_{p}$-sch\'ema lisse.
On d\'efinit l'hypersurface cubique   $\mathcal{X}  \subset  \P^4_{A}$ par l'\'equation
  $$ \Norm_{K/k}(u+\alpha v+\alpha^2w) + xy(x-y) + \pi \Phi(u,v,w,x,y)=0.$$

La fibre g\'en\'erique $\mathcal{X} \times_{A}k$ est une hypersurface cubique lisse dans $\P^4_{k}$.

La fibre sp\'eciale de $\mathcal{X}/A$ 
est d\'efinie sur le corps fini $\F$ par
l'\'equation
$$ \Norm_{E/\F}(u+\beta v+\beta^2 w) + xy(x-y) =0.$$
Ceci est une hypersurface cubique  $Y$ dans $\P^4_{\F}$ qui 
g\'eom\'etriquement a trois points singuliers,  
avec $x=y=0$, conjugu\'es entre eux sous l'action du groupe de Galois de $E/\F$,
donn\'es par l'annulation de deux des conjugu\'es de $u+\beta v+\beta^2w$.
Ceci d\'efinit un unique point ferm\'e $m$ de $Y$.
Soit   $Y_{1} \to Y$ l'\'eclatement de ce point ferm\'e.
En passant sur le corps $E$, on v\'erifie que  l'image r\'eciproque de $m$
est l'union de deux $E$-surfaces lisses $E$-rationnelles, d'intersection  
une   courbe  $L\simeq \P^1_{E}$.
La fl\`eche  $Y_{1}\to Y$ est donc un $CH_{0}$-isomorphisme universel de $\F$-vari\'et\'es.
La vari\'et\'e $Y_{1}$
 a trois points singuliers $P_{i}, i=1,2,3$, 
 de corps
r\'esiduel $E$,
situ\'es sur $L$. 
Soit $Z \to Y_{1}$ l'\'eclatement des   trois points $P_{i}$.
On v\'erifie que $Z$ est lisse sur $\F$, et que l'image r\'eciproque de
chaque $P_{i}$ est une $E$-surface projective lisse $E$-rationnelle.
La fl\`eche $Z \to Y_{1}$ est donc un $CH_{0}$-isomorphisme universel de $\F$-vari\'et\'es.
Ainsi la fl\`eche compos\'ee $Z \to Y_{1} \to Y$ est une d\'esingularisation
qui est un $CH_{0}$-isomorphisme universel de $\F$-vari\'et\'es.

Un calcul un peu \'elabor\'e, renvoy\'e en appendice (Proposition \ref{brauernonramifienontrivial}),
montre que l'on a
    $\Br(Z) \neq 0$, et plus pr\'ecis\'ement que la classe $\xi$  de l'alg\`ebre cyclique $(E/\F, x/y) \in \Br(\F(Y))$
 appartient \`a $\Br(Y)$ et est non nulle.

Le th\'eor\`eme \ref{principal} assure alors que la $k$-vari\'et\'e $X=\mathcal{X}\otimes_{A}k$ 
n'est pas universellement $CH_{0}$-triviale, en particulier elle n'est pas r\'etracte rationnelle.
 \end{proof}

 \begin{rema} 
C'est  une question ouverte si le groupe de Chow $A_{0}(X)$ des cycles de degr\'e z\'ero
 est nul pour toute  hypersurface cubique lisse  $X$ de dimension   3 sur un corps $p$-adique.
 Pour $p\neq 3$ et $k=\Q_{p}$, 
   Esnault et Wittenberg \cite[Ex. 2.10]{EW} ont \'etabli la nullit\'e de $A_{0}(X)$ pour l'hypersurface de $\P^4_{k}$  d'\'equation
   $$x^3+y^3+z^3+pu^3+p^2v^3=0.$$

   \end{rema}

   \begin{rema}
   Sur tout corps de nombres $k$, le th\'eor\`eme \ref{padiquestable} permet de donner
  des exemples d'hypersurfaces cubiques dans $\P^4_{k}$,
avec un point $k$-rationnel,  qui ne sont pas r\'etractes rationnelles sur $k$. 
De fa\c con analogue, on construit de tels exemples
 sur  $k=\F((x))$ puis sur  $k=\F(x)$, avec $\F$ un corps fini.
 \end{rema}

\section{D\'ecomposition de la diagonale en famille et action des correspondances}\label{topologique}

  \subsection{Sur le lieu de d\'ecomposition de la diagonale}
\begin{lem}\label{lieuxgeneral}
Soit $B$ un sch\'ema int\`egre de type fini sur un corps non d\'enombrable $k$. Soit $\overline{k}$ une cl\^oture alg\'ebrique de $k$.
Soit $\{B_i\}_{i\in \N}$ une famille d\'enombrable de sous-sch\'emas ferm\'es de $B$. Les assertions suivantes sont \'equivalentes:

(i) un des $B_{i}$ co\"{\i}ncide avec $B$;

(ii) $\bigcup_i B_i$ contient le point g\'en\'erique de $B$;

(iii) $B(\bar k)\subset \bigcup_i B_i(\bar k)$;

(iv) $\bigcup_i B_i$ contient un point g\'en\'eral de $B$: il existe un ouvert $U\subset B$ tel que $U(\bar k)\subset \bigcup_i B_i(\bar k)$;

(v) $\bigcup_i B_i$ contient un point tr\`es g\'en\'eral de $B$: il existe une famille d\'enombrable $\{F_j\}_j$ de ferm\'es stricts de $B$, telle que $(B(\bar k)\setminus \bigcup_j F_j(\bar k))\subset \bigcup_i B_i(\bar k)$.

\end{lem}

\begin{proof}
Les \'enonc\'es  (i)   et  (ii)   sont \'equivalents.
Les implications 
 (i) $\Rightarrow$  (iii) $\Rightarrow$ (iv) $\Rightarrow$  (v)  sont \'evidentes. Supposons  (v). On a alors $B(\bar k)=  \bigcup_j F_j(\bar k)\cup \bigcup_i B_i(\bar k)$ avec $F_j\subset B$ des ferm\'es stricts. Comme $\bar k$ est non d\'enombrable, cela implique qu'un des $B_i$ co\"incide avec $B$, d'o\`u  (i).
\end{proof}

Soit $k$ un corps alg\'ebriquement clos.  Soit  $X$ une vari\'et\'e propre  int\`egre sur $k$, de dimension $n$, et  soit $x \in X_{lisse}(k)$.
Comme rappel\'e au \S 1 (Lemme \ref{equivalencedeuxpointsdevue}   et  remarque subs\'equente),
 on dit que $X$ {\it admet une d\'ecomposition de Chow de la diagonale} s'il existe un diviseur $D\subset X$ et  un cycle $Z \in 
Z_{n}(X\times X)$ \`a support dans
$D\times X$   tels que
\begin{equation}\label{defidd}
[\Delta_X]=[Z]  + [X\times x]\text{  dans }CH_n(X\times X),
\end{equation}
 o\`u pour un cycle $V$ dans $X$ on \'ecrit $[V]$ pour la classe de $V$ dans $CH(X)$.

On a le th\'eor\`eme suivant (cf. \cite[Theorem 1.1]{voisin} et appendice B ci-dessous).

\thmd\label{specialisationdiagonale}{{\it
 Soit $B$ un sch\'ema int\`egre de type fini sur un corps alg\'ebriquement clos $k$ de caract\'eristique z\'ero. Soit $X\to B$ un morphisme projectif qui admet une section $\sigma:B\to X$.
Il existe une famille d\'enombrable $\{B_i\}_{i}$ de sous-sch\'emas ferm\'es de $B$ telle que, pour tout point $b\in B(k)$, on a : \\

$b\in\bigcup_i B_i(k)\Leftrightarrow X_b$ admet une d\'ecomposition de Chow de la diagonale.\\
}}

En appliquant le lemme \ref{lieuxgeneral}, on obtient

\begin{theo}\label{iddfamille}
Soit $B$ un sch\'ema int\`egre de type fini sur un corps alg\'ebriquement clos $k$. Soit $X$ un $k$-sch\'ema int\`egre et
  $p : X\to B$ un morphisme dominant projectif.
Supposons qu'il existe un $k$-point $b_{0}\in B$, tel que la fibre $X_{b_{0}}$ n'admet pas de d\'ecomposition de Chow de la diagonale. Alors pour $b\in B$ un point tr\`es g\'en\'eral, la fibre $X_b$ n'admet pas  de d\'ecomposition de Chow de la diagonale.
\end{theo}
\begin{proof}On effectue le changement de base $X \to B$, c'est-\`a-dire que l'on consid\`ere la seconde projection
$q : X'=X\times_{B}X \to B'=X$. Il y a ici une section \'evidente. En appliquant la proposition pr\'ec\'edente \`a $X' \to B'$, on trouve une union d\'enombrable
de ferm\'es $G_{i}$ de $X'$ tels qu'il y ait d\'ecomposition de la diagonale pour $X'_{c}$, $c$ point ferm\'e de $B'$ si et seulement si
$c$ est dans l'un des $G_{i}$. Comme $X'_{c}= X_{p(c)}$,  ensemblistement la r\'eunion des $G_{i}$ 
co\"{\i}ncide avec la r\'eunion des images r\'eciproques $p^{-1}(p(G_{i}))$, et pour un point ferm\'e  $b \in B$,
il y a d\'ecomposition de la diagonale pour $X_{b}$ si et seulement si $b$  appartient 
\`a la
  r\'eunion
des ferm\'es $p(G_{i})$. L'hypoth\`ese sur $b_{0}$ assure qu'aucun des $p(G_{i})$ ne co\^{\i}ncide avec $B$.
 \end{proof}

\subsection{D\'ecomposition de la diagonale et action des correspondances sur la cohomologie de Betti d'un solide}

\begin{prop}\label{nontorsion}
Soit $Y$ un solide projectif int\`egre d\'efini sur le corps $\mathbb C$, qui admet une d\'ecomposition de Chow de la diagonale. Supposons qu'il existe  une d\'esin\-gu\-larisation $\pi:  \tilde Y \to Y$ de $Y$ telle que  pour toute composante $E_i$ du  diviseur exceptionnel $E$ de $\pi$  le groupe de cohomologie de Betti $H^2(\tilde E_i, \mathbb Z)$ d'une d\'esingularisation $\tilde E_i$ de $E_i$ n'a pas de torsion.

 Alors les groupes de cohomologie de Betti  $H^4(\tilde Y ,\Z)$ et $H^3(\tilde Y ,\Z)$
 sont sans torsion.
\end{prop}

\begin{proof}
D'apr\`es les hypoth\`eses,  on peut  \'ecrire une \'equivalence rationnelle de   cycles  sur $\tilde Y \times_{k} \tilde Y:$
$$\Delta_{\tilde Y } \equiv Z+ \tilde Y \times x +Z_1+Z_2  $$
o\`u $Z\subset \tilde Y  \times \tilde Y $ est \`a support dans $D\times \tilde Y $ pour $D\subset \tilde Y $ un diviseur,
$Z_1$ est \`a support dans $E\times \tilde Y $ et $Z_2$ est \`a support dans $\tilde Y \times E$. Comme la vari\'et\'e $\tilde Y $ est lisse,  on dispose d'une action des correspondances sur  $H^4(\tilde Y ,\mathbb Z)$.

 Soit $\mathcal Z$ une composante irr\'eductible de $Z_2$; l'image de $\mathcal Z$ par la deuxi\`eme projection est contenue dans un ferm\'e $F\subseteq E_i$ de codimension $c\geq 0$, pour un  certain $i$. Soit $\pi_F:\tilde F\to F$ une d\'esingularisation de $F$. Comme $\mathcal Z$ domine $F$, on peut supposer que $\mathcal Z=\pi_{F,*}\mathcal Z'$ pour 
un certain
 $\mathcal Z'\subset \tilde Y \times \tilde F$. Montrons que l'action de $\mathcal Z$ se factorise par  le groupe $H^{2-2c}(\tilde F,\mathbb Z)$.
Soit $\iota:\tilde F\to \tilde Y $. On a 
le diagramme commutatif suivant, o\`u le carr\'e du milieu commute d'apr\`es la formule de projection (cf. \cite[Chap. 5, \S 6]{spanier}):

\footnotesize
$$\xymatrix{ H^4(\tilde Y , \Z) \ar[r]^(0.4){pr_1^*}\ar@{=}[d]& H^4(\tilde Y \times \tilde F, \Z)\ar[r]^(0.5){\cdot [\mathcal Z']} & H^{8-2c}(\tilde Y \times \tilde F, \Z)\ar[r]^{ pr_{2,*}}\ar[d]^{\iota_*}&H^{2-2c}(\tilde F, \Z)\ar[d]^(0.6){\iota_*}\\
H^4(\tilde Y , \Z) \ar[r]^(0.4){ pr_1^*}&H^4(\tilde Y \times \tilde Y , \Z)\ar[r]^(0.5){\cdot \iota_*[\mathcal Z']} \ar[u]^{\iota^*}& H^{10}(\tilde Y \times  \tilde Y , \Z)\ar[r]^(0.6){pr_{2,*}}&H^{4}(\tilde Y , \Z).}$$
\normalsize
Le groupe $H^{2-2c}(\tilde F,\mathbb Z)$ n'a pas de torsion : si $c=0$, c'est l'hypoth\`ese de la proposition, et si $c>0$,
soit  ce groupe est nul, soit c'est  le groupe $H^0$ d'une courbe lisse et il est donc isomorphe \`a $\Z$. L'application $Z_{2,*}$
est donc nulle sur $H^4(\tilde Y ,\mathbb Z)_{tors}$. De m\^eme, l'application $[\tilde Y\times x]_*$ est nulle sur $H^4(\tilde Y ,\mathbb Z)_{tors}$.

De la m\^eme mani\`ere, d'apr\`es le diagramme commutatif ci-dessous l'action de chaque composante de $Z_1$  qui est \`a support dans $E_i\times \tilde Y $ se factorise par le groupe $H^4(\tilde F,\mathbb Z)$ pour $F\subseteq E_i$ 
et $\tilde F\to F$ une d\'esingularisation de $F$ :
\footnotesize
$$\xymatrix{ H^4(\tilde F, \Z) \ar[r]^(0.4){pr_1^*}& H^4(\tilde F \times \tilde Y , \Z)\ar[r]^(0.5){\cdot  [\mathcal Z']} & H^{8-2c}(\tilde F\times \tilde Y, \Z)\ar[r]^{ pr_{2,*}}\ar[d]^{\iota_*}&H^{4}(\tilde Y , \Z)\ar@{=}[d]\\
H^4(\tilde Y , \Z) \ar[r]^(0.4){ pr_1^*}\ar[u]^{\iota^*}&H^4(\tilde Y \times \tilde Y , \Z)\ar[r]^(0.5){\cdot \iota_* [\mathcal Z']} \ar[u]^{\iota^*}& H^{10}(\tilde Y \times  \tilde Y , \Z)\ar[r]^(0.6){pr_{2,*}}&H^{4}(\tilde Y , \Z).}$$
\normalsize
Le groupe $H^4(\tilde F,\mathbb Z)$  est  isomorphe \`a $\mathbb Z$ ou nul, donc sans torsion, car   
 $\tilde F$   ou bien est  une surface lisse, ou bien    est de dimension au plus un. L'application $Z_{1,*}$
est donc nulle sur $H^4(\tilde Y ,\mathbb Z)_{tors}$.
De m\^eme, l'application $Z_*$ se factorise par un groupe  $H^4(\tilde F,\mathbb Z)$ pour $F\subseteq D$ un ferm\'e et $\tilde F\to F$ une d\'esingularisation de $F$; l'application $Z_*$ est donc nulle sur  le groupe de torsion $H^4(\tilde Y ,\mathbb Z)_{tors}$. 

Comme l'application $\Delta_{\tilde Y _{*}}$ est l'identit\'e,  le groupe  $H^4(\tilde Y ,\mathbb Z)_{tors}$ est nul.
Il en est donc de m\^eme de $H^3(\tilde Y ,\mathbb Z)_{tors}$.
\end{proof}

\begin{rema}
On notera que les hypoth\`eses faites sur le diviseur exceptionnel de la d\'esingularisation $\pi:  \tilde Y \to Y$ sont
a priori plus faibles que la $CH_{0}$-trivialit\'e universelle demand\'ee au \S \ref{algebrique}.
\end{rema}

 \begin{theo}\label{quartiques2}
 Dans l'espace projectif $\P^N$ param\'etrant les hypersurfaces quartiques dans $\P^4_{\C}$, un point tr\`es g\'en\'eral correspond \`a une hypersurface quartique lisse qui n'est pas r\'etracte rationnelle, et  qui en particulier n'est pas stablement rationnelle.
\end{theo}
\begin{proof}

Dans l'appendice A, nous partons d'un solide quartique $V$, d\'efini par une \'equation
\begin{equation}
\label{V}
 \alpha(z_0,z_1,z_2)z_3^2+\beta(z_0,z_1,z_2)z_3+\gamma(z_0,z_1,z_2)+z_0^2z_4^2=0,
\end{equation}
birationnel \`a un solide d'Artin-Mumford \cite{artinmumford}. Nous construisons une d\'esin\-gu\-larisation
$W\to  V$ satisfaisant :

  (i) Les composantes irr\'eductibles du diviseur  exceptionnel $E$ de $W \to V$  sont des  surfaces rationnelles 
  universellement $CH_{0}$-triviales (non n\'ecessairement lisses).

 (ii) le groupe $H^4(W,\Z)_{tors} \simeq  H^3(W,\Z)_{tors}$ est non nul (ceci est donn\'e par \cite{artinmumford}).

On applique le th\'eor\`eme \ref{iddfamille} avec $B=\P^N$ et $b_0$ le point correspondant \`a une quartique $V$ comme dans (\ref{V}).  D'apr\`es la 
proposition \ref{nontorsion}, la quartique $V$ n'a pas de d\'ecomposition de Chow de la diagonale car le groupe $H^3(W,\Z)_{tors} \simeq \Br(W)$ est non nul. L'\'enonc\'e suit alors des lemmes \ref{retractCh0trivial} et  \ref{equivalencedeuxpointsdevue}~a).
\end{proof}

\appendix
\section{R\'esolution des singularit\'es d'une quartique  $V \subset \P^4_{\overline{\Q}}$ 
birationnelle \`a un solide d'Artin--Mumford.}

Soit $A\subset \mathbb P^2$ une conique lisse, d\'efinie par une \'equation $\alpha(z_0,z_1,z_2)=0$.
Soient $E_1, E_2\subset  \mathbb P^2$ deux courbes elliptiques lisses d\'efinies par des \'equations $\epsilon_1(z_0,z_1, z_2)=0$ et $\epsilon_2(z_0, z_1, z_2)=0$, chacune tangente \`a $A$ en trois points, les points de tangence \'etant tous distincts,
et telles que les courbes $E_{1}$ et $E_{2}$ s'intersectent en 9 points, deux \`a deux distincts, et distincts des pr\'ec\'edents.

D'apr\`es \cite{artinmumford}, sur $\overline{\Q} \subset \C$,
 il existe deux formes homog\`enes $\beta(z_0, z_1, z_2)$ et $\gamma(z_0, z_1, z_2)$, de degr\'es respectifs $3$ et $4$, telles que $$\beta^2-4\alpha\gamma=\epsilon_1\epsilon_2.$$

\begin{lem}\label{surfacequartique}
Soit $S\subset\mathbb P^3_{\overline{\Q}}$ la surface quartique d\'efinie
par l'\'equation homog\`ene
\begin{equation}\label{S}
g=\alpha(z_0,z_1,z_2)z_3^2+\beta(z_0,z_1,z_2)z_3+\gamma(z_0, z_1, z_2)=0.
\end{equation}

 (i) Les singularit\'es de la surface $S$ sont des 
singularit\'es quadratiques 
 ordinaires
  : le point $P_0=(0:0:0:1)$ et les neuf points dont les projections depuis $P_0$ sur le plan $z_3=0$ sont les points de $E_1\cap E_2$.

 (ii) L'ensemble $M=\{g=0, \frac{\partial g}{\partial z_1}=0, \frac{\partial g}{\partial z_3}=0\}\cup \{g=0, \frac{\partial g}{\partial z_2}=0, \frac{\partial g}{\partial z_3}=0\}$ est fini.
\end{lem}

\begin{proof}
La partie  (i)  est dans \cite{artinmumford}. Montrons  (ii).
Soit $Q=(z_0:z_1:z_2:z_3)\in M\setminus\{P_0\}$. Soit $q=(z_0:z_1:z_2)$ la projection de $Q$ sur le plan $z_3=0$. Supposons  $\frac{\partial g}{\partial z_1}(Q)=0$, le deuxi\`eme cas est identique. La condition $\frac{\partial g}{\partial z_3}(Q)=0$ donne $\beta(q)=-2z_3\alpha(q)$. Montrons que l'on a  $\alpha(q)\neq 0$. Sinon, la condition pr\'ec\'edente donne $\beta(q)=0$ et donc $\gamma(q)=0$
d'apr\`es l'\'equation. Ainsi la multiplicit\'e de $\epsilon_1\epsilon_2$ en $q$ est $2$. Comme $E_1$ et $E_2$ sont lisses, on en d\'eduit que $Q\in E_1\cap E_2$ et on obtient une contradiction car $A$ ne passe pas par  $E_1\cap E_2$. On a donc $\alpha(q)\neq 0$ et $z_3=-\frac{\beta(q)}{2\alpha(q)}$.

On \'ecrit
\begin{equation}\label{alphag}
4\alpha\cdot g=(2z_3\alpha+\beta)^2-\epsilon_1\epsilon_2.
\end{equation}

Comme $\beta(q)=-2z_3\alpha(q)$, on a donc $\epsilon_1\epsilon_2(q)=0$. On peut supposer $\epsilon_1(q)=0$, le deuxi\`eme cas est similaire.  Comme $\alpha(q)\neq 0$, on a $z_3=-\frac{\beta(q)}{2\alpha(q)}$. En d\'erivant l'\'equation (\ref{alphag}) par rapport \`a $z_1$, on d\'eduit  $\frac{\partial g}{\partial z_1}(Q)=0\Rightarrow \frac{\partial \epsilon_1\epsilon_2}{\partial z_1}(q)=0$. Comme $\epsilon_1(q)=0$, on a soit  $\epsilon_2(q)=0$, soit $\frac{\partial \epsilon_1}{\partial z_1}(q)=0$.   Dans le premier cas,   $q$ appartient \`a l'ensemble fini $E_1\cap E_2$. Sinon $q$  appartient \`a l'ensemble $\{\epsilon_1=0, \frac{\partial \epsilon_1}{\partial z_1}=0\}$, qui est fini car $E_1$ est une courbe elliptique lisse.

\end{proof}

Comme l'ensemble $M\setminus\{P_0\}$ est fini, quitte \`a 
 faire
un changement lin\'eaire en les coordonn\'ees $z_0, z_1, z_2$, on peut supposer dans la suite :

\begin{multline}\label{bonhyperplan}
\text{ L'hyperplan }z_0=0\text{ ne contient aucun point de }M\setminus\{P_0\}   \text{ ni aucun point}  \\ \text{de l'intersection de }A \text{\ et }E_1\cup E_2\text{, et  il n'est pas tangent \`a la conique }A.
\end{multline}

Soit $V\subset \mathbb P^4_{\overline{\Q}}$ une quartique d\'efinie par l'\'equation homog\`ene
 \begin{equation}
f=\alpha(z_0,z_1,z_2)z_3^2+\beta(z_0,z_1,z_2)z_3+\gamma(z_0,z_1,z_2)+z_0^2z_4^2=0.
 \end{equation}

Soit $L\subset \mathbb P^4$ la droite d'\'equation $z_0=z_1=z_2=0$.

 \begin{lem}
  Les  points singuliers de la quartique $V$ qui ne sont pas situ\'es sur la droite $L$ sont  $9$  
  singularit\'es quadratiques 
 ordinaires.
 \end{lem}
 \begin{proof}
 Soit $Q=(z_0:z_1:z_2:z_3:z_4)\in V\setminus L$ un point singulier. Si $z_0=0$, alors  le point $(z_0:z_1:z_2:z_3)$ est un point singulier de $S$. D'apr\`es le choix de l'hyperplan $z_0=0$, c'est le point $(0:0:0:1)\in L$. On a donc $z_0\neq 0$. On a $\frac{\partial f}{\partial z_4}=2z_0^2z_4=0$, d'o\`u $z_4=0$ et le point $(z_0:z_1:z_2:z_3)$ correspond encore \`a un point singulier de la surface $S$ d\'efinie par (\ref{S}). Comme la surface $S$ n'a que
 des singularit\'es quadratiques  
  ordinaires
  d'apr\`es le lemme \ref{surfacequartique}, on a des coordonn\'ees locales $\mathfrak z_i$, $i=1\ldots 3$ telles que l'\'equation locale de $g$ s'\'ecrive comme une somme de $\sum_{i=1}^3\mathfrak z_i^2$ et de termes de plus haut degr\'e. Comme $z_0\neq 0$, on a les coordonn\'ees locales $\mathfrak z_1,\ldots \mathfrak z_3, z_4$ de $Q$ et l'\'equation de $f$ s'\'ecrit comme  $\sum_{i=1}^3\mathfrak z_i^2+z_4^2+\text{termes de degr\'e sup\'erieur}$;   on a donc  que $Q$ est un point singulier quadratique
   ordinaire.
 \end{proof}

 Soit $V'\to V$ l'\'eclatement de la droite $L\subset V$. Soit $P\in L$ le point $z_3=0$.

 \begin{lem}
  La vari\'et\'e $V'$ est lisse en tout point de $V'$ au-dessus de $L\setminus P$. Le diviseur exceptionnel $E'$ de $V'\to V$ est une surface rationnelle, les fibres de l'application $E'\to L$ sont des coniques et la fibre g\'en\'erique est lisse.
 \end{lem}
 \begin{proof}
 Soit $Q'\in V'$ un point singulier au-dessus d'un point $Q=(0:0:0:z_3:z_4)\in L$. On a donc soit $z_3\neq 0$, soit $z_4\neq 0$.
 \begin{enumerate}
  \item Si $z_4\neq 0$, on a une \'equation de $V$ dans les coordonn\'ees affines $x_i=\frac{z_i}{z_4}$, $0\leq i\leq 3$
  $$\alpha(x_0,x_1,x_2)x_3^2+\beta(x_0,x_1,x_2)x_3+\gamma(x_0,x_1,x_2)+x_0^2=0.$$
  Comme les r\^oles de $x_1$ et $x_2$ sont sym\'etriques, on n'a que deux cartes de l'\'eclatement \`a consid\'erer :
  \begin{enumerate}
  \item $x_0=y_0$, $x_1=y_1y_0, x_2=y_2y_0$, $x_3=y_3$. L'\'equation de $V'$ s'\'ecrit
  $$\alpha(1,y_1, y_2)y_3^2+\beta(1, y_1, y_2)y_0y_3+\gamma(1, y_1, y_2)y_0^2+1=0.$$
  Pour tout point de $V'$ au-dessus de $L$ on a $y_0=0$ et on obtient l'\'equation de $E'$
   $$\alpha(1,y_1, y_2)y_3^2+1=0.$$
  Les fibres au-dessus des points de $L$ sont donc des coniques et la fibre au-dessus du point g\'en\'erique 
  $\Spec \C(y_3)$ 
  de $L$ est lisse.

   Si $Q'=(0,y_1,y_2,y_3)$ est un point singulier de $V'$ au-dessus d'un point de $L$, on a  donc $\alpha(Q')y_3^2+1=0$ de l'\'equation ci-dessus et $\alpha(Q')\cdot 2 y_3=0,$ en d\'erivant par rapport \`a $z_3$, ce qui n'est pas possible.

    \item $x_0=y_0y_1$, $x_1=y_1$, $x_2=y_1y_2$, $x_3=y_3$. L'\'equation de $V'$ s'\'ecrit
    \begin{equation}\label{cartesinguliere}
  \alpha(y_0, 1, y_2)y_3^2+\beta(y_0, 1, y_2)y_1y_3+\gamma(y_0, 1, y_2)y_1^2+y_0^2=0.
  \end{equation}
   Pour tout point de $V'$ au-dessus de $L$ on a 
   $y_1=0$ et l'\'equation de $E'$ s'\'ecrit
   $$\alpha(y_0, 1, y_2)y_3^2+y_0^2=0.$$
  Les fibres au-dessus des points de $L$ sont donc des coniques. Pour un point singulier de la fibre au-dessus du point g\'en\'erique $\Spec\C(y_3)$ de $L$, on v\'erifie d'abord que $y_0=0$, puis que $(0:1:y_2)$ donne un point singulier de $A$ (d\'efini sur  le corps $\C(y_3)$), ce qui n'est pas possible. La fibre g\'en\'erique  est donc une conique lisse.

  Soit $Q'=(y_0,0,y_2,y_3)$ un point singulier de $V'$.
  \begin{enumerate} \item Si $y_3\neq 0$, on a $\alpha(Q')\cdot 2y_3=0$, d'o\`u $\alpha(Q')=0$. On d\'eduit de l'\'equation de $V'$ que $y_0=0$ et puis que $\frac{\partial\alpha(y_0, 1, y_2)}{\partial y_i}= \frac{\partial\alpha}{\partial z_i}(y_0, 1, y_2)=0$ pour $i=0, 2$. Comme 
  $\alpha(0, 1, y_2)=0$, on en d\'eduit que le point $(0:1:y_2)$ est un point singulier de $A$, contradiction.
  \item Si $y_3=0$, alors $y_0=0$ de l'\'equation de $V'$.  On v\'erifie qu'effectivement tous les points de la droite $y_0=y_1=y_3=0$ sont des points singuliers.
  \end{enumerate}
     \end{enumerate}
  \item Supposons   $z_3\neq 0$. On a donc l'\'equation de $V$ en  coordonn\'ees affines :
  $$\alpha(x_0,x_1,x_2)+\beta(x_0,x_1,x_2)+\gamma(x_0,x_1,x_2)+x_0^2x_4^2=0.$$
  Il suffit d'analyser deux cartes de l'\'eclatement:
  \begin{enumerate}
  \item $x_0=y_0$, $x_1=y_1y_0, x_2=y_2y_0$, $x_4=y_4$. L'\'equation de $V'$ s'\'ecrit
  $$\alpha(1,y_1, y_2)+\beta(1, y_1, y_2)y_0+\gamma(1, y_1, y_2)y_0^2+y_4^2=0.$$
   Pour tout point de $V'$ au-dessus de $L$, on a $y_0=0$, et l'\'equation de $E'$ s'\'ecrit
   $$\alpha(1, y_1, y_2)+y_4^2=0.$$
   Le fibres au-dessus des points de $L$ sont donc des coniques et la fibre au-dessus du point g\'en\'erique $\Spec \C(y_4)$ de $L$ est lisse.

    Si $Q'=(0,y_1,y_2,y_4)$ est un point singulier de $V'$ au-dessus d'un point de $L$, on a donc $2y_4=0$, d'o\`u $y_4=0$ et puis $\alpha(Q')=0$. Comme dans le cas $1(b)(i)$, on montre que $Q'$ est un point singulier de $A$, contradiction.
   \item $x_0=y_0y_1$, $x_1=y_1, x_2=y_2y_1$, $x_4=y_4$. L'\'equation de $V'$ s'\'ecrit
   $$\alpha(y_0, 1, y_2)+\beta(y_0, 1, y_2)y_1+\gamma(y_0, 1, y_2)y_1^2+y_0^2y_4^2=0$$ et au-dessus de $L$ on a $y_1=0$ et on obtient l'\'equation
   $$\alpha(y_0, 1, y_2)+y_0^2y_4^2=0.$$
   Ici encore pour un point singulier au-dessus du point g\'en\'erique
    $\Spec \C(y_4)$ de $L$
     on voit d'abord que $y_0=0$, puis que le point $(0:1:y_2)$ correspond \`a un point singulier de $A$ (d\'efini sur le corps $\C(y_4)$), contradiction.

    Si $Q'=(y_0, 0,y_2,y_4)$ est un point singulier de $V'$, on a $2y_4y_0^2=0$. On a donc $\alpha(Q)=0$ et $\frac{\partial\alpha(y_0, 1, y_2)}{\partial y_i}=0, i=0,2$, on obtient donc un point singulier de $A$, contradiction.

  \end{enumerate}
 \end{enumerate}
   Comme la surface $E'$ est fibr\'ee en coniques au-dessus d'une droite, le th\'eor\`eme de  Max Noether, ou de Tsen, montre que  c'est une
  surface rationnelle.

 \end{proof}

  Soit $L'\subset V'$ la droite image r\'eciproque du point $P$. Soit $V''\to V'$ l'\'eclatement de la droite $L'$.

 \begin{lem}
 Les seules singularit\'es de $V''$ sont des 
 singularit\'es quadratiques 
  ordinaires. La vari\'et\'e $V''$ est lisse en tout point au-dessus de $L$.  Le diviseur exceptionnel $E''$ de $V''\to V'$ est une surface rationnelle lisse et  les fibres de l'application $E''\to L'$ sont des coniques.
 \end{lem}
 \begin{proof}

 On reprend l'\'equation de la carte singuli\`ere (\ref{cartesinguliere}) de $V'$ :
  $$ \alpha(y_0, 1, y_2)y_3^2+\beta(y_0, 1, y_2)y_1y_3+\gamma(y_0, 1, y_2)y_1^2+y_0^2=0.$$

  Notons qu'on a une deuxi\`eme carte singuli\`ere, en \'echangeant les r\^oles de $x_1$ et $x_2$ dans le lemme pr\'ec\'edent.

  On a trois cartes affines dans l'\'eclatement $V''$ de la droite $L': y_0=y_1=y_3=0$ au-dessus de la carte (\ref{cartesinguliere}).
  \begin{enumerate}
   \item $y_0=u_0, y_1=u_0u_1, y_2=u_2, y_3=u_0u_3$. L'\'equation de $V''$ s'\'ecrit
   $$\alpha(u_0,1,u_2)u_3^2+\beta(u_0,1,u_2)u_1u_3+\gamma(u_0,1,u_2)u_1^2+1=0.$$
 Pour tout point au-dessus de $L'$ on a $u_0=0$ et obtient l'\'equation de $E''$
  $$\alpha(0,1,u_2)u_3^2+\beta(0,1,u_2)u_1u_3+\gamma(0,1,u_2)u_1^2+1=0.$$

  Pour un point singulier  de $E''$, on a $\alpha(0,1,u_2)\cdot 2u_3+\beta(0,1,u_2)\cdot u_1=0$ et  $\beta(0,1,u_2)\cdot u_3+\gamma(0,1,u_2)\cdot 2u_1=0$, d'o\`u    
  $$\alpha(0,1,u_2)u_3^2+\beta(0,1,u_2)u_1u_3+\gamma(0,1,u_2)u_1^2=0\neq -1.$$
De m\^eme, pour un point singulier $Q''$ de $V''$ on montre que $$\alpha(Q'')u_3^2+\beta(Q'')u_1u_3+\gamma(Q'')u_1^2=0\neq -1.$$
\item $y_0=u_0u_1, y_1=u_1, y_2=u_2, y_3=u_1u_3$.  L'\'equation de $V''$ s'\'ecrit
   $$\alpha(u_0u_1,1,u_2)u_3^2+\beta(u_0u_1,1,u_2)u_3+\gamma(u_0u_1,1,u_2)+u_0^2=0.$$
   Pour tout point au-dessus de $L'$ on a $u_1=0$ et on obtient l'\'equation de $E''$
   $$\alpha(0,1,u_2)u_3^2+\beta(0,1,u_2)u_3+\gamma(0,1,u_2)+u_0^2=0.$$
   Pour un point singulier de $E''$ on a $2u_0=0$, donc $$\alpha(0,1,u_2)u_3^2+\beta(0,1,u_2)u_3+\gamma(0,1,u_2)=0.$$  Puis on obtient  $2\alpha(0,1,u_2)u_3+\beta(0,1,u_2)=0$ en d\'erivant par rapport \`a $u_3$. Comme $\frac{\partial\alpha(0, 1, u_2)}{\partial u_2}=\frac{\partial\alpha}{\partial z_3}(0, 1, u_2)$ et de m\^eme pour $\beta$ et $\gamma$, en d\'erivant par rapport \`a $z_2$, on d\'eduit que 
     $(0:1:u_2:u_3)\in  M\setminus\{P_0\}$, ce qui n'est pas possible d'apr\`es l'hypoth\`ese~\ref{bonhyperplan}.

   Pour un point singulier $Q''$ de $V''$ on a une analyse similaire :  $\frac{\partial\alpha(u_0u_1, 1, u_2)}{\partial u_0}=0$ et de m\^eme pour $\beta$ et $\gamma$, d'o\`u $u_0=0$ en prenant la d\'eriv\'ee  par rapport \`a $u_0$. En prenant la d\'eriv\'ee par rapport \`a $u_2$  on en d\'eduit que la projection $(0:1:u_2:u_3)$ de $Q''$ est dans $ M\setminus\{P_0\}$, contradiction avec le choix de l'hyperplan $z_0=0$.
\item $y_0=u_0u_3, y_1=u_1u_3, y_2=u_2, y_3=u_3$.  L'\'equation de $V''$ s'\'ecrit
   $$\alpha(u_0u_3,1,u_2)+\beta(u_0u_3,1,u_2)u_1+\gamma(u_0u_3,1,u_2)u_1^2+u_0^2=0.$$
    Ici, pour  tout point au-dessus de $L'$ on a $u_3=0$ et on obtient l'\'equation de $E''$
   $$\alpha(0,1,u_2)+\beta(0,1,u_2)u_1+\gamma(0,1,u_2)u_1^2+u_0^2=0.$$
    Pour un point singulier de $E''$ on a $2u_0=0$, d'o\`u  $$\alpha(0,1,u_2)+\beta(0,1,u_2)u_1+\gamma(0,1,u_2)u_1^2=0.$$ Si $u_1\neq 0$, on montre de la m\^eme fa\c{c}on que dans la carte pr\'ec\'edente que 
    $(0:1:u_2:\frac{1}{u_1})\in M\setminus P_0$ et on obtient une contradiction. Si $u_1=0$, on a $\alpha(0,1,u_2)=0$ de l'\'equation, et puis $\beta(0,1,u_2)=-2u_1\gamma(0,1,u_2)=0$. On a donc  
    $(0:1:u_2)\in A\cap(E_1\cup E_2)$, contradiction avec l'hypoth\`ese \ref{bonhyperplan}.

     Pour $Q''$ un point singulier de $V''$ on a  $\frac{\partial\alpha(u_0u_3, 1, u_2)}{\partial u_0}=0$ et de m\^eme pour $\beta$ et $\gamma$, d'o\`u $u_0=0$.  On d\'eduit de m\^eme que soit $u_1\neq 0$ et $(0:1:u_2:u_1^{-1})\in M\setminus\{P_0\}$,  soit $u_1=0$ et 
     $(0:1:u_2)\in A\cap(E_1\cup E_2)$, ce qui n'est pas possible d'apr\`es le choix de l'hyperplan $z_0=0$.
  \end{enumerate}
  Comme la surface $E''$ est fibr\'ee en coniques au-dessus d'une droite, le th\'eor\`eme de Max Noether, ou de Tsen, montre que c'est une
  surface rationnelle.
  
\end{proof}

   Soit $W\to V''$ l'\'eclatement des 
   singularit\'es quadratiques 
  ordinaires de $V''$.

 \begin{prop}\label{lemodele}

 (i) La $\overline{\Q}$-vari\'et\'e $W$ est projective et lisse.

 (ii) Le diviseur  exceptionnel $E$ de la r\'esolution $\pi: W\to V$ s'\'ecrit comme l'union disjointe $E=\sqcup_{i=1}^{10} E_i$, o\`u :

  (iia) les  composantes $E_1,\ldots E_9$ sont des quadriques lisses au-dessus des points 
  singuliers  quadratiques 
   ordinaires;

  (iib) la composante $E_{10}$ est l'union de deux surfaces   $E'\cup E''$,   le morphisme $\pi$ induit  une fibration $E'\to L$ dont les fibres sont des coniques et la fibre g\'en\'erique est lisse et poss\`ede un point rationnel, la surface $E''$ est rationnelle lisse et   $\pi(E'')=P$.

(iii) Le morphisme $W\to V$ est un $CH_0$-isomorphisme universel.

 (iv) Le groupe $H^4(W_{\C},\Z)_{tors} \simeq H^3(W_{\C},\Z)_{tors}  \simeq \Br\,W_{\C} $ est non nul.
  \end{prop}
 \begin{proof}
 Les propri\'et\'es (i) et (ii)  r\'esultent de la construction et des lemmes pr\'ec\'edents.
D'apr\`es  la proposition \ref{chowisolissesing}, pour \'etablir (iii) il suffit de v\'erifier que sur tout corps $F$  les fibres de $W_{F}\to V_{F}$ au-dessus des $F$-points  sont
universellement
 $CH_{0}$-triviales, ce qui r\'esulte de $(ii)$ car ces fibres sont soit la $F$-surface rationnelle projective et lisse $E''_{F}$, soit une $F$-conique lisse avec un point rationnel, 
 soit
une conique r\'eductible sur $F$, soit une $F$-surface quadrique lisse d\'eploy\'ee.
  Puisque $W$ est birationnelle \`a la vari\'et\'e d'Artin-Mumford, la propri\'et\'e (iv) r\'esulte de \cite{artinmumford}.
 \end{proof}

\section{Le lieu de la d\'ecomposition de la diagonale}

Cet appendice donne des d\'etails sur  la preuve du th\'eor\`eme suivant (\cite[Theorem 1.1 et Proposition 1.4]{voisin}).

\thmd\label{Den}{{\it Soit $k$ un corps alg\'ebriquement clos de caract\'eristique z\'ero. Soient $B$ un $k$-sch\'ema lisse et $X\to B$ un  morphisme projectif qui admet une section $\sigma:B\to X$.  Alors  il existe une famille d\'enombrable $\{B_i\}_{i \in \N}$,  
de
sous-sch\'emas ferm\'es de $B$, telle que  
$$\{b\in B(k) \,| \, X_b\mbox{ admet une d\'ecomposition de Chow de la diagonale}\}=\cup_{i \in \N} B_i(k).$$}}

Pour d\'emontrer ce  th\'eor\`eme, on utilise un ``argument  de sch\'emas de Hilbert''. Plus pr\'ecis\'ement, on consid\`ere des sch\'emas qui param\`etrent des sous-sch\'emas et des sch\'emas qui param\`etrent des cycles sur $X$, ainsi que les familles universelles correspondantes. Cela  
 utilise des r\'esultats profonds d'existence de sch\'emas  $Hilb$ et $Chow$.

 {\it Dans tout cet appendice,  $k$ est un corps alg\'ebriquement clos de caract\'eristique z\'ero.}\\

 Rappelons d'abord la d\'efinition des sch\'emas de Chow  (voir \cite[Def I.3.11; Thm. I.3.2.1]{kollar}). Soient $S$ un $k$-sch\'ema et  $X/S$ un sch\'ema sur $S$. Soit $W/S$ un sch\'ema r\'eduit. Une {\bf famille bien d\'efinie 
 de cycles alg\'ebriques propres} de $X/S$ sur $W/S$  est la donn\'ee de
 \begin{enumerate}
 \item  un cycle $U=\sum m_i [U_i]$, i.e. $U_i$ sont des sch\'emas int\`egres, avec $Supp\, U\subset X\times_S W$, $m_{i} \in \mathbb Z $.
 \item $g:Supp\,U\to W$ propre, tel que chaque fibre de $g_i:=g|_{U_i}$ est soit vide, soit de dimension $d$, et tel que l'image  de $g_i$ est une composante irr\'eductible de $W$.
 \end{enumerate}
 Dans le cas o\`u sch\'ema $W$ n'est pas normal, on impose une condition suppl\'ementaire dont on n'aura pas besoin ici.
 
 Si $X/S$ est un sch\'ema projectif, alors (\cite[Thm. I.3.2.1]{kollar})
 on dispose d'une famille de $S$-sch\'emas projectifs et semi-normaux $Chow_{X/S}^{d,d'}$ (param\'etr\'es par des donn\'ees d\'enombrables de degr\'es $(d,d')$) qui repr\'esente le foncteur des cycles {\bf non-n\'egatifs}, sur la cat\'egorie des $S$-sch\'emas semi-normaux. Dans la suite on va souvent omettre les indices $(d, d')$. On a aussi 
 une  famille universelle  $Univ_{X/S}\to Chow_{X/S}$.\\

 \subsection{Lemmes pr\'eliminaires}
 \lem\label{suppDiv}{ Soient $B$ un $k$-sch\'ema lisse et $X\to B$ un  morphisme projectif. Il existe une famille d\'enombrable $\mathcal F$  de  $B$-sch\'emas
 lisses $F_i$, $i\in \N$  \'equip\'es d'une famille 
de
cycles $V_i/B$ sur  $F_i/B$ de $Y=X\times_B X/B$, telle que, pour tout $b\in B(k)$ 
et tout  
cycle effectif  $Z$ dans $Y_b$ de dimension $d$ dont l'image par la premi\`ere projection est
incluse
 dans le support d'un diviseur de $X_b$, il existe un point $x\in F_i$ au-dessus de $b$ tel que $Z$ s'identifie \`a $V_{i,x}$.}
\proof{Soit $H$ une composante 
du sch\'ema
 de Hilbert de $X/B$ qui param\`etre les diviseurs  effectifs (plats et de codimension relative 1), $D'\to H$ la famille universelle  
 et $D := D' \times_{B} X$
  (les sous-sch\'emas de $Y$ qui sont   produits d'un diviseur sur $X$ et  de $X$) : comme c'est une famille projective, on dispose des sch\'emas de Chow de $D/H$.  Soient $C$ une  composante  du sch\'ema de Chow $Chow_{D/H}$ des cycles de dimension $d$ et $V$  la famille universelle,  on a  $Supp V\subset C\times_H D$ : la fibre de $V$ au-dessus de  $c\in C$ d'image $h\in H$ est  un cycle inclus dans $D_h$. La famille $V$ peut \^etre aussi vue comme une famille 
 de
  cycles $V/B$
   param\'etr\'ee
   par $C/B$ dans $Y/B$ : en effet, $Supp V\subset C\times_H D\subset C\times _H (Y\times_B H)= C\times _B Y.$  
     Soient $\tilde{C}$
   une  d\'esingularisation (Hironaka) de $C$  et     $\tilde{V}$ le tir\'e-en-arri\`ere de la famille $V$ (cela ne change pas les fibres). On trouve une famille   d\'enombrable  $\mathcal F =  \{\tilde{C}\}$    (o\`u l'union est sur toutes les composantes $H$ et sur toutes les composantes $C$ de  $Chow_{D/H}$.)\\
\qed}

\lem{Soient $B$ un $k$-sch\'ema lisse  et $Y\to B$ un  morphisme projectif. Il existe une famille d\'enombrable de sch\'emas normaux $T_i, i\in \N$ munis d'une famille de cycles $\mathcal D_i\to T_i$ telle que, pour tout $b\in B$ 
et pour toute sous-vari\'et\'e int\`egre $W$ de dimension $d+1$ dans $Y_b$, il existe une d\'esingularisation $\tilde W$ de $W$ telle que
pour tout diviseur de fonction $D=D_1-D_2$ sur $\tilde W$, avec $D_1, D_2$ effectifs, il existe $i$ et un point $t\in T_i(k)$ au-dessus de $b$ tel que la fibre $\mathcal D_{i, t}$ s'identifie \`a $(D_1, D_2)$.}
\proof{Soit $G$ un sch\'ema quasi-projectif qui est une composante irr\'eductible, munie de sa structure r\'eduite,
 de l'union des sch\'emas (cf.  \cite[9.7.7]{EGAIV}) qui param\`etrent les sous-vari\'et\'es int\`egres de $Y/B$ (plats sur $B$) de dimension $d+1$ et soit $W\to G$ une famille universelle : c'est un morphisme projectif, plat, \`a fibres g\'eom\'etriques int\`egres.  
La fibre g\'en\'erique $W_{k(G)}$ est int\`egre. 
Soit $\tilde W_{k(G)}$ une r\'esolution des singularit\'es de $W_{k(G)}$. On peut supposer que le morphisme $\tilde W_{k(G)}\to W_{k(G)}$ s'\'etend sur un ouvert $G_1$ de $G$. Quitte \`a changer $G_1$ par un ouvert plus petit, on peut supposer que pour tout $t\in G_1(k)$ la fibre $\tilde W_t$ est lisse et birationnelle \`a $W_t$, puisque ces propri\'et\'es sont v\'erifi\'ees 
au point g\'en\'erique. On pose $W_1=W|_{G_1}$ et $\tilde W_1=\tilde W_{G_1}$. Par r\'ecurrence (sur $dim\, G$), on peut donc trouver une d\'ecomposition $G=\cup_{j=1}^m G_j$ avec $G_j$ des sch\'emas quasi-projectifs,  localement ferm\'es dans $G$, des familles $\tilde W_j\to G_j$ \`a fibres  projectives lisses, telles que pout tout $t\in G(k)$, la fibre $\tilde W_{j, t}$ est une r\'esolution de $W_t$. Quitte \`a changer $G_j$ par une r\'esolution et $\tilde W_{j}$ par le tir\'e-en-arri\`ere de la famille $W_j$, on peut m\^eme supposer que les  $G_j$ sont lisses.

D'apr\`es \cite[FGA  revisited, Chap. 9,  3.7 et 4.8]{kleiman}), sous les hypoth\`eses que 
$\tilde W_j/G_j$ est projectif, plat, \`a fibres g\'eom\'etriques int\`egres, on dispose alors
 de sch\'emas $\Div_{\tilde W_j/G_j}$ (qui param\`etre les diviseurs de Cartier effecifs) et $\Pic_{\tilde W_j/G_j}$, munis de familles universelles.  On a en plus un morphisme $Ab: \Div_{X/S}\to \Pic_{X/S}.$
Soit $\Delta_j$,  vu comme $B$-sch\'ema,    l'image r\'eciproque par le morphisme $(Ab, Ab): \Div_{\tilde W_j/G_j}\times \Div_{\tilde W_j/G_j} \to \Pic_{\tilde W_j/G_j}\times \Pic_{\tilde W_j/G_j}$ de la diagonale. On a alors que $\Delta_j$ est l'union (finie) de 
 ses composantes irr\'eductibles. Soit $T$ une de ces composantes (munie de sa structure r\'eduite) et soit $\tilde T$  la normalisation de $T$. Soit $\mathcal T$ la famille $\mathcal T=\{\tilde T\}$ o\`u l'on prend l'union sur toutes les composantes $G$, sur tout $j$ correspondant \`a la stratification de $G$ comme ci-dessus et sur toutes les composantes $\tilde T$ obtenues \`a partir de $\Delta_j$.
  Alors $T_i\in \mathcal T$ et $\mathcal D_i$  la famille universelle induite par celle sur les sch\'emas $\Div$ conviennent. Un
point
 $t\in T_i$ au-dessus de $b\in B$ correspond \`a
 un sous-sch\'ema 
 int\`egre de dimension $d+1$
 dans $Y_b$ et un diviseur d'une fonction sur une r\'esolution de ce sch\'ema.\qed\\}

\lem\label{famDiv}{Soient $B$ un $k$-sch\'ema lisse  et $Y\to B$ un  morphisme projectif. Il existe une  famille d\'enombrable de  sch\'emas normaux $H_i, i\in \N$ munis d'une famille de cycles $\mathcal S_i$ sur $H_i$ telle que, pour tout $b\in B$ et pour tout cycle $D=\sum_{i=1}^nD_{i1}-D_{i2}$ de dimension $d$ dans $Y_b$, avec $D_{i1}, D_{i2}$ effectifs, qui sont des diviseurs d'une fonction sur  une sous-vari\'et\'e int\`egre $W_{it}$ de dimension $d+1$ dans $Y_t$, il existe $i$ et un point $t\in H_i(k)$ au-dessus de $b$ tel que la fibre $\mathcal S_{i, t}$ s'identifie \`a $\cup (D_{i1} \cup D_{i2})$. }
\proof{ On prend la r\'eunion d\'enombrable sur tous les $n$-uplets $(T_1,\ldots, T_n)$ avec $T_i$ dans la famille du lemme pr\'ec\'edent. On pose $H_i$ le normalis\'e du produit $\prod_{j=1}^nT_j$ (avec $T_j$ comme dans le lemme pr\'ec\'edent) et  $S_i$ l'union des tir\'es-en-arri\`ere des $\mathcal D_j\to T_j$. \qed\\} 

Le lemme ci-dessus permet de montrer

\prop\label{fu}{ Soient $B$ un $k$-sch\'ema lisse  et $Y\to B$ un  morphisme projectif.  Soit $Z=(Z_1, Z_2)\in Chow_{Y/B}(B)\times Chow_{Y/B}(B)$ deux
familles bien d\'efinies  de 
cycles alg\'ebriques propres effectifs,
 de dimension $d$, au  sens de Koll\'ar.  Alors  il existe une famille d\'enombrable $\pi_i:M_i\to B$ ($i\in \N$) de sch\'emas quasi-projectifs au-dessus de $B$  munis de familles $U_i\to M_i$ projectives, 
telle que :
\begin{itemize}
\item[i)] L'union des images 
$B_i : = \pi_i(M_i(k))$ 
dans $B(k)$ est exactement le lieu
$$\{b\in B(k) \,| \, Z_{1, b}-Z_{2,b}\mbox{ est rationnellement \'equivalent \`a z\'ero dans } Y_b.\}$$
\item[ii)] Pour tout point $b\in B(k)$ et pour toute donn\'ee 
$S=(W_{i,b}, D_{i1}, D_{i2})_{i=1}^n$
avec $n \geq 1$ un entier, $W_{ib}$ des sous-sch\'emas int\`egres de dimension $d+1$
 de $Y_b$,
 $D_{i1}, D_{i2}$ deux diviseurs effectifs de Weil sur (la normalisation de) $W_{ib}$ tels que $D_{i1}-D_{i2}$ est 
 le diviseur d'une fonction rationnelle sur  $W_{ib}$,
  tel que l'on ait  l'\'egalit\'e de
cycles $$Z_{1b}+\sum D_{i1}= Z_{2b}+\sum D_{i2}\mbox{ dans } Chow_{Y/B}(\kappa(b))$$ il existe $l$ et un point $t\in M_i$ au-dessus de $b$ tel que la fibre $U_t$ s'identifie \`a la donn\'ee  $S$.\\ 
\end{itemize}
}
\proof{On prend $H_i$ comme dans le lemme \ref{famDiv}, avec la famille $S_i$ des cycles sur $H_i$ dans $Y/B$.  Soit $Z'=(Z'_1, Z'_2)\in Chow_{Y/B}(H_i)\times Chow_{Y/B}(H_i)$ la famille 'constante' $Z_{H_i}$. 
Par la construction de la famille  $\mathcal S_i$, elle param\`etre les couples $(C_1, C_2)$ de  cycles effectifs, dont la diff\'erence est rationnellement \'equivalente \`a z\'ero (est somme de  diviseurs de  fonctions).
Soit $M_i\subset H_i$ l'image r\'eciproque de la diagonale dans $Chow_{Y/B}\times Chow_{Y/B}$ par le morphisme 
\begin{gather*}H_i\to Chow_{Y/B}\times Chow_{Y/B}\\
(C_1, C_2)\mapsto (Z_1+C_1, Z_2+C_2).
\end{gather*}
Autrement dit, soit $\mathcal S_i\cup Z'$ une famille de cycles sur $H_i$ dans $Y/B$ et soit $H_i\to Chow_{Y/B}\times Chow_{Y/B}$ le morphisme induit. On d\'efinit $M_i$ comme le sous-sch\'ema ferm\'e de $H_i$, l'image r\'eciproque de la diagonale dans $Chow_{Y/B}\times Chow_{Y/B}$.
Alors $M_i$ avec la famille universelle induite par $\mathcal S_i$ convient.
\qed\\}

\subsection{Preuve du th\'eor\`eme} 

La proposition \ref{fu} permet de montrer l'\'enonc\'e suivant (cf. \cite[Prop 1.4]{voisin}) 

\prop\label{1.4}{Soient $B$ un $k$-sch\'ema lisse  et $Y\to B$ un  morphisme projectif. Soit $Z=(Z_1, Z_2)\in Chow_{Y/B}(B)\times Chow_{Y/B}(B)$ deux 
familles bien d\'efinies  de 
cycles alg\'ebriques propres, de dimension $d$, 
au sens
de Koll\'ar.  Alors  il existe une famille d\'enombrable $\{B_i\}_{i\in \N}$, de sous-sch\'emas ferm\'es de $B$, telle que  
$$\{b\in B(k) \,| \, Z_{1, b}-Z_{2,b}\mbox{ est rationnellement \'equivalent \`a z\'ero dans } Y_b\}=\cup B_i(k).$$}

\proof{Notons qu'ici  $Z_{i, b}$ 
est
soit vide, soit de dimension $d$ dans $Y_b$, d'apr\`es la d\'efinition d'un cycle bien d\'efini.\\
L'implication $\ref{fu}\Rightarrow \ref{1.4}$ suit la preuve dans \cite{voisin} (pp. 8-9), on ne la refait pas ici. 
Elle utilise la sp\'ecialisation de Fulton, comme dans les arguments du \S 1.\\
}

{\it D\'emonstration du th\'eor\`eme \ref{Den}}. 
On  d\'eduit le th\'eor\`eme \ref{Den} de la proposition \ref{1.4}. Soit $F_i$ une composante  et $V_i$ la famille universelle comme dans le lemme \ref{suppDiv}  pour $d=dim\,X$.  Soit $F_i^2=F_i\times _B F_i$.  Soient $V_i'=V_i\cup (\Delta_X)_{F_i}$  et $V_i''=V_i\cup (X\times \sigma(B))_{F_i}$ les translat\'es de cette famille. On applique la proposition \ref{1.4} \`a $Y=(X\times_B X)\times_B F_i^2$ et $Z_1=V'_i\times F_i$, $Z_2=F_i\times V''_i$ et  $B=F_i^2$. 
Si $t=(t_1, t_2)\in F_{i}^2$ au-dessus de $b$, la fibre de $Z_1$ en $t$ s'identifie au cycle effectif $\Delta_{X_b}+z_1$ o\`u $pr_{1*}Supp (z_1)$ est un sous-sch\'ema propre de $X_t$, de codimension au moins $1$ et la fibre de $Z_2$ en $t$ s'identifie au cycle effectif $X_b\times \sigma(b)+z_2$ o\`u $pr_{1*}Supp (z_2)$ est un sous-sch\'ema propre de $X_t$, de codimension au moins $1$.

On trouve une union d\'enombrable de ferm\'es $B_i$ comme dans la proposition \ref{1.4}, qui correspondent au lieu o\`u l'on a  l'\'egalit\'e $\Delta_{X_b}+z_1=X_b\times \sigma(b)+z_2$. Ensuite on prend encore l'union d\'enombrable sur tous les $i$ (correspondant
aux $F_i$).\qed

\section{Un calcul de groupe de Brauer}

\begin{prop}\label{brauernonramifienontrivial}
Soit $K/k$ une extension cubique galoisienne de corps,
de groupe de Galois $G$.
Supposons $k$ de caract\'eristique $p\neq  3$.
Soit $K=k(\theta)$. Soit $X \subset \P^5_{k}$ l'hypersurface cubique
donn\'ee par l'\'equation homog\`ene
$$\Norm_{K/k}(u+v\theta + w \theta^2)=xy(x-y).$$
Pour tout $k$-mod\`ele projectif et lisse $Y$ de $X$,  le groupe   $\Br(Y)/\Br(k)$
est la somme directe de $\Z/3$ et d'un groupe de torsion $p$-primaire.
En particulier $X$ n'est pas r\'etracte $k$-rationnelle.
\end{prop}

\medskip
\begin{proof}

On consid\`ere d'abord le cas d\'eploy\'e, qu'on peut d\'efinir par l'\'equation homog\`ene
$$uvw=xy(x-y).$$
 Les  trois $k$-points d\'efinis par l'annulation de $x,y$ et de deux
des trois coordonn\'ees $(u,v,w)$ sont les seuls points non lisses de $X$.
Soit $U \subset X$ leur compl\'ementaire.

Soit $V$ le compl\'ementaire dans $X$ du ferm\'e d\'efini par $xy=0$.

On a $V\subset U$, le compl\'ementaire \'etant form\'e des 6 diviseurs
g\'eom\'etriquement  irr\'eductibles 
$\Delta_{u,x}$, 
$\Delta_{v,x}$, $\Delta_{w,x}$,
$\Delta_{u,y}$, $\Delta_{v,y}$, $\Delta_{w,y}$, 
o\`u par exemple $\Delta_{u,x}$
est  le diviseur sur $U$ d\'efini par $u=x=0$.

En coordonn\'ees affines $(u,v,w,x)$, la $k$-vari\'et\'e $V$
est    d\'efinie par :
$$ uvw=x(x-1),  \hskip2mm x \neq 0.$$
Soit $p = V \to \G_{m,k}$ la fl\`eche donn\'ee par la coordonn\'ee $x$.
La fibre g\'en\'erique de $p$ est un $k(x)$-tore d\'eploy\'e. Son groupe
de Picard est nul, et les fonctions inversibles sur cette fibre sont 
de forme $f(x)u^nv^m$ avec $f(x) \in k(x)$ et $n,m \in \Z$.
Les fibres de $p$ au-dessus des points ferm\'es distincts de $x=1$
 sont toutes g\'eom\'etriquement int\`egres, leur classe dans le groupe de Picard
 de $W$ est nulle. La seule fibre non int\`egre est donn\'ee par $x=1$,
 c'est la somme de trois diviseurs, chacun est principal, car d\'efini par
 $u=0$, resp. $v=0$, resp. $w=0$. Comme le groupe de Picard de $\G_{m}$
 est nul, on conclut $\Pic(V)=0$. Par ailleurs, une fonction inversible
 sur $V$ est de la forme $cx^r(x-1)^su^nv^m$ avec $c \in k^{\times}$ et $r,s,n,m \in \Z$.
 La consid\'eration de son diviseur montre que l'on a $s=n=m=0$.
 On a donc \'etabli que la suite 
  $$1 \to  k[W]^{\times}/k[U]^{\times} \to \Div_{U \setminus W}(U) \to \Pic(U) \to \Pic(V)$$
 est une suite exacte de r\'eseaux :
$$ 0 \to \Z \to  \Z \Delta_{u,y} \oplus  \Z \Delta_{v,y} \oplus \Delta_{w,y} \oplus
\Z \Delta_{u,x} \oplus  \Z \Delta_{v,x} \oplus \Z \Delta_{w,x} \to \Pic(U) \to 0,$$
la fl\`eche issue de $\Z$ envoyant $1$, classe de la fonction rationnelle $y/x \in k(U)$ sur 
$$  \Delta_{u,y} + \Delta_{v,y} + \Delta_{w,y} -
 \Delta_{u,x} - \Delta_{v,x} -\Delta_{w,x}.$$

 Soit maintenant $K/k$ extension galoisienne de corps comme dans l'\'enonc\'e.
Soit $\overline k$ une cl\^oture s\'eparable de $k$ contenant $K$. Soit $g=Gal({\overline k}/k)$
 et $h=Gal({\overline k}/K)$.
Soit $G=Gal(K/k) \simeq \Z/3$.
Les points non lisses de l'hypersurface cubique $X_{K}$ sont les trois points
d'intersection des 3 droites d\'efinies dans $x=y=0$ par 
$\Norm_{K/k}(u+v\theta + w\theta^2)=0$. 

Soit $U$ la $k$-vari\'et\'e  compl\'ementaire  de ces trois points.
Sur $K=k(\theta)$, on retrouve la situation d\'eploy\'ee, avec $K^{\times}=K[U]^{\times}$.
Soit $V$ le compl\'ementaire de $xy=0$.
Soit $D_{x,K}$ le diviseur irr\'eductible d\'efini sur $U_{K}$ 
par $u+\theta v+ \theta ^2 w=0,  x=0$
et $D_{y,K}$ le diviseur d\'efini par $u+\theta v+ \theta ^2 w=0, y=0$.
 La suite exacte
 ci-dessus donne alors une suite exacte de $G$-r\'eseaux
  $$ 1 \to K[V]^{\times}/K^{\times}  \to \Div_{U_{K} \setminus V_{K}} (U_{K}) \to  \Pic(U_{K}) \to 0. $$
soit 
 $$0 \to \Z \to \Z[G] \oplus \Z[G] \to \Pic(U_{K}) \to 0,$$
 la fl\`eche $ \Z \to \Z[G] \oplus \Z[G] \subset \Div(U_{K})$ envoyant $1$ sur la classe du diviseur de la fonction rationnelle $(y/x)$ sur son diviseur  $$div_{U_{K}}(y/x)= \Norm_{K/k}(D_{y,K}) -  \Norm_{K/k}(D_{y,K}).$$
 
 De cette suite  on d\'eduit un isomorphisme 
 $$\hat{H}^{-1}(G,\Pic(U_{K})) \simeq   \hat{H}^0(G,K[V]^{\times}/K^{\times}) =  \hat{H}^0(G,\Z)= \Z/3,$$
 le groupe $ \hat{H}^0(G,K[V]^{\times}/K^{\times})$ \'etant engendr\'e par la classe de $y/x \in k[V]^{\times}$.

On a le diagramme de suites exactes de $G$-modules

\begin{equation}
\xymatrix@C=16pt{
0 \ar[r] & K[V]^{\times}/K^{\times}   \ar[r]\ar[d] & \Div_{U_{K}\setminus V_{K}}(U_{K} )  \ar[r] \ar[d]&   \Pic(U_{K})   \ar[r] \ar[d]^{=}  & 0 \\
0 \ar[r] & K(U)^{\times}/K^{\times}   \ar[r] & \Div(U_{K})  \ar[r] &   \Pic(U_{K})   \ar[r]  & 0
}
\end{equation}

Par le lemme de Shapiro, $\hat{H}^{-1}(G,P)=0$ pour tout $G$-module de permutation $P$,
en particulier pour les $G$-modules $\Div_{U_{K}\setminus V_{K}}(U_{K}) $ et  $\Div(U_{K})$.

On a donc  les suites exactes compatibles

\xymatrix@C=16pt{
0 \ar[r] & \hat{H}^{-1}(G,\Pic(U_{K}))     \ar[r]^{\simeq}  \ar[d]^{=} &   \hat{H}^0(G, K[V]^{\times}/K^{\times}   )  \ar[d] & \\
0 \ar[r] & \hat{H}^{-1}(G,\Pic(U_{K}))   \ar[r] &  \hat{H}^0(G, K(U)^{\times}/K^{\times}) \ar[r] &  \hat{H}^0(G, \Div(U_{K})  )   
}

Il r\'esulte de ce diagramme que l'image de $y/x$ dans $ \hat{H}^0(G, K(U)^{\times}/K^{\times})= 
 \hat{H}^0(G, K(Y)^{\times}/K^{\times})$ est {\it non \'egale \`a $1$}.

Par ailleurs,   un argument valuatif 
sur l'\'equation de la vari\'et\'e $X$ montre que
 que la fonction $y/x \in k(X)^{\times}$ a sur tout mod\`ele   lisse $Y$ de $V$ son
diviseur qui est la norme d'un diviseur sur $Y_{K} $ pour l'extension $K/k$.

Observons que $X$ poss\`ede le $k$-point lisse $u=v=w=x=0, y=1$.

On conclut alors (\cite[Lemme 14 et Lemme 15]{CTS}) que la classe de l'alg\`ebre cyclique $(K/k, y/x) \in \Br (k(X))=\Br(k(Y))$
est non ramifi\'ee, et qu'elle est non constante.

On peut aussi observer  que  la suite exacte de $G$-modules attach\'ee \`a un mod\`ele projectif et lisse $Y$
donne naissance \`a la suite exacte
$$ 0 \to \hat{H}^{-1}(G,\Pic(Y_{K})) \to \hat{H}^0(G, K(Y)^{\times}/K^{\times} )  \to  \hat{H}^0(G, \Div(Y_{K})).$$
et que l'on a ainsi $\hat{H}^{-1}(G,\Pic(Y_{K}))\neq 0$,
 et donc $ H^{1}(G,\Pic(Y_{K})) \neq 0$. Comme $Y_{K}$ est $K$-rationnelle, ce qui est clair sur l'\'equation de $X$,
 on a un isomorphisme
 $ H^1(G,\Pic(Y_{K})) \simeq H^1(k,\Pic(X\times_{k}{\overline k}))$,
 o\`u $\overline k$ est une cl\^oture s\'eparable de $k$ contenant $K$.
 L\`a encore, l'existence d'un $k$-point lisse sur $X$ permet de conclure $$\Br(Y)/\Br(k) \neq 0.$$
 
 Pour \'etablir l'\'egalit\'e plus pr\'ecise $\Br(Y)/\Br(k) = \Z/3$ \`a la torsion $p$-primaire pr\`es,
 on observe que  comme $Y_{K}$ est $K$-rationnelle, le quotient   $\Br(Y_{K})/\Br(K) $ 
 est un groupe de torsion $p$-primaire, et que l'on a  l'inclusion
 $$\Br(Y)/\Br(k)  \subset \Br(U)/\Br(k).$$
 L'argument ci-dessus montre que la classe $(K/k,y/x)$ d'une part appartient
 \`a $\Br(Y)/\Br(k)$, d'autre part engendre le noyau de $\Br(U)/\Br(k) \to \Br(U_{K})/\Br(K)$,
 qui est isomorphe \`a $H^1(G,\Pic(U_{K})) \simeq \Z/3$.
 \end{proof}
 
 \medskip

\rema{
Une fois \'etabli  $K^{\times}=K[U]^{\times}$ et
$H^1(G,\Pic(U_{K}))\simeq \Z/3$,
on peut aussi \'etablir $H^1(G,\Pic(Y_{K}) )\neq 0$ en supposant qu'on conna\^{\i}t
une d\'esingularisation $Y \to X$ induisant un isomorphisme au-dessus de $U$.
On alors  la suite exacte de $G$-modules
$$ 0 \to \Div(Y_{K}\setminus U_{K}) \to \Pic(Y_{K}) \to \Pic(U_{K}) \to 0.$$
Comme $Y\setminus U$ est au-dessus des trois points singuliers  
d\'efinis sur $K$ et conjugu\'es transitivement par l'action de $G$, 
on voit que le $G$-module de permutation $\Div(Y_{K})$
est forc\'ement une somme directe d'exemplaires de $\Z[G]$.
Mais alors la suite exacte donne
$$H^1(G,\Pic(Y_{K})) \simeq H^1(G,\Pic(U_{K}))\simeq \Z/3.$$}

\end{document}